\documentclass[10pt]{amsart}

\usepackage[T1]{fontenc}
\usepackage[utf8]{inputenc}

\usepackage[left=2.7cm,right=2.7cm,top=3.2cm,bottom=3.2cm]{geometry}

\usepackage{mathtools}      
\usepackage{amsthm}
\usepackage{amssymb}        
\usepackage{mathrsfs}
\usepackage{dsfont}
\usepackage{extarrows}
\usepackage{empheq}
\usepackage{cases}

\usepackage{xcolor}
\usepackage{enumitem}
\usepackage{comment}

\usepackage{natbib}

\usepackage{graphicx}
\graphicspath{{fig/}}

\usepackage{algorithm}
\usepackage{algpseudocode}

\usepackage[
  colorlinks=true,
  urlcolor=blue,
  citecolor=blue,
  linkcolor=blue,
  linktocpage,
  pdfpagelabels,
  bookmarksnumbered,
  bookmarksopen
]{hyperref}

\numberwithin{equation}{section}
\mathtoolsset{showonlyrefs}
\theoremstyle{plain}
\newtheorem{theorem}{Theorem}[section]
\newtheorem{lemma}[theorem]{Lemma}
\newtheorem{corollary}[theorem]{Corollary}

\theoremstyle{definition}
\newtheorem{definition}[theorem]{Definition}

\theoremstyle{remark}
\newtheorem{remark}[theorem]{Remark}

\newtheorem{theoremletter}{Theorem}

\newcommand{\papertitle}
{A Brezis--Merle inequality for the biharmonic operator with Dirichlet boundary conditions in four-dimensional balls}

\begin{document}


\begin{center}

{\Large\bfseries \papertitle\par}

\vspace{0.8cm}

{\scshape Jo\~ao Marcos do \'O}\\
{\small Department of Mathematics, Federal University of Para\'{\i}ba, Brazil}\\
{\small \texttt{jmbo@mat.ufpb.br}}

\vspace{0.4cm}

{\scshape Federica Sani}\\
{\small Dipartimento di Scienze Fisiche, Informatiche e Matematiche,}\\
{\small Universit\`a di Modena e Reggio Emilia, Via Campi 213/A, 41125 Modena, Italy}\\
{\small \texttt{federica.sani@unimore.it}}

\vspace{0.4cm}

{\scshape Victor Rafael Silva}\\
{\small Department of Mathematics, Federal University of Para\'{\i}ba, Brazil}\\
{\small \texttt{victor.rafael@academico.ufpb.br}}

\vspace{0.5cm}

{\small \today}

\end{center}

\vspace{0.5cm}


\begin{abstract}
In this paper, we establish a Brezis-Merle type inequality for
very weak solutions of the biharmonic Dirichlet problem in the
four-dimensional unit ball. By combining this inequality with moving-plane
arguments, we obtain uniform a priori estimates for positive solutions
of semilinear biharmonic equations with exponential nonlinearities in
both the subcritical and critical regimes. As an application of these
estimates, we  prove the existence of a
positive solution by means of topological methods. We also extend the analysis to a class of
non-variational cooperative systems.
\end{abstract}

\makeatletter
\@setabstract
\makeatother

\vspace{0.3cm}


\medskip



    \tableofcontents

	\section{Introduction}

The goal of this work is to prove uniform estimates for {\it very weak solutions} of a class of fourth-order equations and systems with Dirichlet boundary conditions in the unit ball  $B\coloneq B(0,1) \subseteq \mathds{R}^4$ involving nonlinearities with exponential growth. Our approach is based
on a new Brezis--Merle type inequality for biharmonic problems. Although our analysis also encompasses systems, for clarity of exposition we will mainly present our results and proofs in the simpler setting of a scalar equation, and subsequently show how the theory that we will develop can be adapted to cover the case of systems. More precisely, we start by considering problems of the form
\begin{equation}\label{App}\tag{$P$}
\left\{
\begin{aligned}
\Delta^{2}u \;&=\; f(x,u) && \text{in } B,\\
u\;&=\; 0 && \text{on } \partial B,\\
\partial_{\nu}u \;&=\; 0 && \text{on } \partial B,\\
u \;&>\; 0 && \text{in } B,
\end{aligned}
\right.
\end{equation}
where $\Delta^{2}$ is the biharmonic operator, the function $f: \overline{B} \times \mathds{R}^{+} \to \mathds{R}^{+}$ is continuous and satisfies suitable growth conditions, and $\nu$ denotes the unit outer normal vector to the boundary of the ball.

Here, solutions of \eqref{App} are understood in the very weak sense, that is, $ u \in L^1(B),$ $f(x,u) \in L^1(B)$ and
\begin{equation}\label{very weak solution}
    \int_{B} u \Delta^2\varphi\,\mathrm{d}x=\int_{B}f(x,u)\varphi\,\mathrm{d}x,\quad \forall \varphi \in C_{\mathrm{cl}}^{4}(\overline{B}),
\end{equation}
where
\begin{equation}\label{Dirichlet test space}
    C_{\mathrm{cl}}^{4}(\overline{B})\coloneq \left\{\varphi \in C^{4}(\overline{B}): \varphi=\partial_{\nu}\varphi=0\quad \text{on}\quad \partial B\right\}.
\end{equation}
The space of test functions  defined in \eqref{Dirichlet test space} is the natural choice to encode the boundary conditions of the Dirichlet biharmonic problem \eqref{App}. Indeed, suppose that
\(u\) is smooth and satisfies
\[
\Delta^{2}u=f \quad \text{in } B,
\qquad
u=\partial_{\nu}u=0 \quad \text{on } \partial B.
\]
Then integration by parts gives
\[
\int_B u\,\Delta^2\varphi\,\mathrm{d}x
=
\int_B f\,\varphi\,\mathrm{d}x \quad \forall\,\varphi \in C_{\mathrm{cl}}^{4
}(\overline{B}),
\]
since the boundary
terms vanish precisely because both \(u\) and \(\varphi\) satisfy the same
Dirichlet boundary conditions.
\medskip

The work of Brezis and Merle \cite{BM} established that very weak solutions of Poisson equations with $L^1$ data in two-dimensional smooth bounded domains satisfy an exponential integrability estimate. This result became an important tool in the study of Liouville-type equations of the form $-\Delta u = V(x)e^u$ and the related analysis of compactness and blow-up phenomena; see, for instance, \cite{LiShafrir1994}. Several works have extended ideas from the Brezis-Merle theory to different settings. In the quasilinear case, problems involving the $N$-Laplacian in dimension $N$ with exponential right-hand side have been studied from this point of view; see, for example, \cite{AguilarCrespoPeral1997,UbillaRufLorca}. For higher-order equations, however, additional difficulties arise, since comparison principles and positivity arguments are no longer available in the same form as in the second-order theory.

For biharmonic equations of the form $\Delta^2 u =f(x,u)$ in $\Omega \subset \mathds R^4$, the boundary conditions have a strong influence on the analysis of the problem. Under Navier boundary conditions,
\begin{equation*}
u=\Delta u=0 \quad \text{on } \partial\Omega,
\end{equation*}
the fourth-order equation can be written as a system of two second-order Dirichlet equations. This makes it possible to use positivity and comparison arguments coming from second-order theory. In this direction, Lin \cite{Lin1998} proved a Brezis--Merle type estimate for the Navier biharmonic problem in dimension four. Related extensions for polyharmonic equations, in conformal dimensions, under Navier boundary conditions were later obtained by Martinazzi in \cite{MartinazziBMNavier}. The Dirichlet boundary conditions
\begin{equation*}
u=\partial_\nu u=0 \quad \text{on } \partial\Omega
\end{equation*}
lead to a different problem. In this case, there is in general no reduction to a second-order system, and the comparison argument used in the original second-order theory cannot be applied directly. Thus, the Dirichlet biharmonic problem raises a natural Brezis-Merle type question: \emph{can one still obtain exponential integrability estimates for the solution itself with only $L^1$ data?} In dimension four, the logarithmic singularity of the biharmonic Green kernel indicates the natural exponential scale for such an estimate.

This question is related to a broader line of works on higher-order elliptic problems in conformal dimensions, where Green function estimates, blow-up analysis, and moving-plane type arguments play a central role. Several recent works develop general frameworks for polyharmonic problems in this direction, treating wide classes of nonlinearities and domains; see, e.g., the work by Mancini and Romani \cite{Mancini} and the references therein. Our contribution has a different emphasis. In the unit ball $B\subset\mathds{R}^4$, we focus on a linear Brezis-Merle type estimate for the \emph{Dirichlet} biharmonic operator with $L^1$ data. Using Boggio's explicit Green function, we obtain such an estimate with sharp exponenti $32\pi^2$. This inequality is then used to prove uniform estimates for solutions of the Dirichlet biharmonic problem \eqref{App} at the level of very weak solutions.

\subsection{Brezis-Merle type inequality}

The main tool in our argument to prove uniform estimates for any very weak solution of \eqref{App} is the following inequality of Brezis-Merle type for solutions of the \emph{linear} problem

\begin{equation}\label{P}
     \tag{$LP$}
\left\{\begin{array}{rcl}
\Delta^2 u  =  f(x) \ \ &\text{in} \ \ B,\\
u  =  \partial_{\nu}u = 0\ \ &\text{on} \ \ \partial B.\\
\end{array}\right.
\end{equation}

\begin{theorem} \label{B-M} If $f \in L^1(B)$ and $u$ is a  very weak solution of \eqref{P}, then for every $\delta \in (0,32\pi^2)$ we have
$$\int_{B} \exp\left[\frac{(32\pi^2-\delta)\lvert u(x) \rvert}{\lVert f \rVert_{L^1(B)}}\right]\,\mathrm{d}x \leq  \frac{256}{\delta}\pi^4.$$
\end{theorem}

The proof of \hyperref[B-M]{Theorem \ref{B-M}} is based on a strategy relying on estimates for the Green function of the biharmonic operator under Dirichlet boundary conditions. This choice is naturally motivated by the higher-order character of the operator, for which comparison principles are in general not available. By contrast, in the second-order case, the original argument of Brezis and Merle \cite{BM} makes essential use of a comparison between a very weak solution satisfying Dirichlet boundary conditions and another very weak solution of the same problem but without boundary conditions; see also \cite[Propositions 5.1 and 5.2]{Ponce}. Such a principle is usually proved by exploiting the positive part of a Sobolev function in \(W^{1,1}\), a mechanism that does not appear to extend in any direct way to higher-order Sobolev spaces. This obstruction prevents a straightforward adaptation of the Brezis-Merle method to our framework and leads us instead to a Green function based approach. Our proof relies on the Green representation formula for solutions of \eqref{P} (see \cite{Boggioclassical,Gazzola}) and Jensen's inequality, avoiding any direct use of a test function transfer principle as in  \cite[Proposition 5.2]{Ponce}.

\subsection{Uniform a priori bounds and existence in the scalar case}
\label{subsec:MainResults}

As a representative application of our main result, we mention classical nonlinearities of the form $f(x,t)=V(x) e^t$. The following theorem states a special case of our more general result (see \hyperref[a priori estimate1 final]{Theorem \ref{a priori estimate1 final}} further below).

\begin{theorem} \label{thm:ModelCase}
Assume that $f(x,t)=V(x)\,e^{t}$, with $V$ satisfying
\begin{equation}\label{V0}
\exists C_1,C_2>0 \text{ such that } C_1\leq V(x)\leq C_2 \quad \forall x \in B.
\end{equation} 
Suppose that there exist $\overline r$, $\overline\delta>0$ such that
\begin{enumerate}[label=(\roman*), leftmargin=*]
\item $V\in C^{1}(\omega_{\overline r})$,
\item for every $x\in \omega_{\overline r}$ and $\theta\in \mathds{S}^{3}$ with $\Bigl|\theta-\frac{x}{|x|}\Bigr|\le \overline\delta,$
we have $\langle \nabla V(x),\theta\rangle \le 0,$ 
\end{enumerate}
where $\omega_{\overline r}\coloneqq \bigl\{x\in B:\mathrm{dist}(x,\partial B)<\overline r\bigr\}$.
Then there exists a constant $C>0$ such that every very weak solution $u$ of \eqref{App} satisfies $\|u\|_{L^{\infty}(B)}\le C$.
\end{theorem}

For the proof of \hyperref[thm:ModelCase]{Theorem \ref{thm:ModelCase}}, we refer to \hyperref[a priori estimate1 final]{Theorem \ref{a priori estimate1 final}} and \hyperref[proof of thm:Model Case]{Remark \ref{proof of thm:Model Case}} below. It is worth pointing out that the classical nonlinear term mentioned above falls within the scope of the analysis developed by Mancini and Romani \cite{Mancini}. The uniform bounds obtained in \cite{Mancini} hold for solutions to biharmonic problems with Dirichlet boundary conditions in the ball (or in more general smooth bounded domains, provided that an energy bound on the solutions is prescribed). Specifically, \cite{Mancini} deals with positive and superlinear nonlinearities satisfying a separation of variables condition at infinity. Since we do not require such a condition, our general result \hyperref[a priori estimate1 final]{Theorem \ref{a priori estimate1 final}} complements the result  \cite{Mancini} in the ball, and it applies to nonlinearities that are not covered by their analysis (see \hyperref[Example]{Section \ref{Example}}).

A further feature of our approach is the level of regularity at which the main \emph{linear} estimate is
established. In the spirit of Brezis-Merle, our exponential integrability inequality is proved at the level
of very weak solutions with merely $L^1$ data. Once the nonlinear term belongs to $L^1$, standard elliptic
bootstrap yields additional regularity; thus, in the semilinear applications we eventually work in the
 weak framework. The novelty is that the core Brezis-Merle estimate is obtained in a very weak $L^1$ formulation and then serves as a
quantitative tool in the subsequent \emph{a priori} bound arguments.

\medskip

Our main results establish an \emph{a priori} \(L^\infty\) bound for very weak solutions of \eqref{App} under the following assumptions on the nonlinear term $f(x,t)$. We will assume that the function $f: \overline{B} \times \mathds{R}^{+} \to \mathds{R}^{+}$ is continuous and satisfies either the 
\emph{Brezis-Merle subcritical}  growth condition 
\begin{flalign}\tag{$f_0$}\label{f1'}
    \exists \, \alpha \in(0,1), \: \exists\, C>0 \text{ such that } f(x,t) \leq Ce^{t^{\alpha}}\quad \forall x \in B, \: t\geq 0,
\end{flalign}
or the \emph{Brezis-Merle critical} exponential growth condition
\begin{equation}\tag{$f_1$}\label{f1}
\begin{aligned}
\exists \, C>0 & \text{ such that }     &f(x,t) \leq Ce^{t}\quad \forall x \in B, \: t\geq 0,\\
\exists \, D,T>0 &\text{ such that }  De^{t}\leq &f(x,t)\quad \hspace{1.0cm}\forall\, x \in B,  \: t\geq T.
\end{aligned}
\end{equation}
In addition, we will assume the spectral lower-growth condition
\begin{flalign}\tag{$f_2$}\label{f2}
     &\liminf_{t\to+\infty}\frac{f(x,t)}{t}> \lambda_1\quad \text{uniformly in }x \in B,
     \end{flalign}
where $\lambda_1$ denotes the first eigenvalue of $(\Delta^2, H_0^{2}(B))$, and it is well known that $\lambda_1>0$ (see e.g. \cite{Gazzola}). Moreover,  we will require the following structural condition: there exist $\overline r$, $\overline\delta>0$ such that
\begin{equation}\tag{$f_3$}\label{f3}
\begin{aligned}
&\text{(i) for every } x\in \omega_{\overline r}, 
\text{ the map } t\mapsto f(x,t) \text{ is increasing on } [0,\infty), \\
&\text{(ii) for every } t\ge 0,\ 
f(\cdot,t)\in C^1(\omega_{\overline r}) \text{ and} \\
&\qquad
\langle \nabla_x f(x,t),\theta\rangle \le 0
\quad \text{for every } x\in \omega_{\overline r},\ t\ge 0,\ 
\theta\in \mathds S^3
\text{ with } |\theta-\nu(x)|\le \overline\delta,
\end{aligned}
\end{equation}
where $\nu(x)\coloneqq x/|x|$ for every $x\in \overline B\setminus\{0\}$, and the set $\omega_{\overline r}$ is a neighborhood of the boundary of $B$, as introduced in the statement of \hyperref[thm:ModelCase]{Theorem \ref{thm:ModelCase}}.

\medskip

Our first general result establishes the \emph{a priori} bound under the subcritical growth condition \eqref{f1'}.

\begin{theorem}\label{a priori estimate2 final} Suppose that \eqref{f1'}, \eqref{f2}, and \eqref{f3} hold. Then there exists a constant \(C>0\) such that every
very weak solution \(u\) of \eqref{App} satisfies
\[
\|u\|_{L^{\infty}(B)}\le C.
\]
\end{theorem}

Our main theorem addresses the critical growth regime, where assumption \eqref{f1'} is replaced by \eqref{f1}.

\begin{theorem}\label{a priori estimate1 final} Suppose that \eqref{f1}, \eqref{f2}, and \eqref{f3} hold. Then there exists a constant \(C>0\) such that every
very weak solution \(u\) of \eqref{App} satisfies
\[
\|u\|_{L^{\infty}(B)}\le C.
\]
\end{theorem}

It is worth noting that exponential control of the growth is, in a precise sense, a natural
threshold at the level of distributional solutions. In the final part of \cite{Mancini}, the authors
construct a counterexample showing that if one allows \emph{non-homogeneous} nonlinearities growing faster than any
$e^{\gamma t}$, then distributional solutions may fail to be bounded (indeed, they build an unbounded
distributional solution). This illustrates that, without additional structure, one cannot expect
uniform $L^\infty$ bounds in genuinely supercritical \emph{and non-homogeneous} regimes.

\begin{remark} \label{proof of thm:Model Case}
\hyperref[thm:ModelCase]{Theorem \ref{thm:ModelCase}} is a particular case of \hyperref[a priori estimate1 final]{Theorem \ref{a priori estimate1 final}}. Indeed, assume that $f(x,t)=V(x)e^t$, with $V$ satisfying \eqref{V0} and $(i)-(ii)$ of \hyperref[thm:ModelCase]{Theorem \ref{thm:ModelCase}}. Then \eqref{V0} yields \eqref{f1}. Moreover,
$f(x,t)$ is superlinear at infinity
uniformly with respect to $x\in B$, and therefore \eqref{f2} holds. Finally, the map
$t\mapsto V(x)e^t$ is increasing and
$\nabla_x f(x,t)=e^t\nabla V(x).$
Hence, the regularity assumption $(i)$ and the monotonicity assumption $(ii)$ give \eqref{f3}.
\end{remark}

The proofs of \hyperref[a priori estimate2 final]{Theorem \ref{a priori estimate2 final}} and \hyperref[a priori estimate1 final]{Theorem \ref{a priori estimate1 final}} rely on some intermediate results of independent interest, which are detailed in \hyperref[app sec 4]{Section \ref{app sec 4}}. First, assuming only an exponential upper growth condition on $f(x,t)$ inspired by the Brezis-Merle inequality, we show that every very weak solution of problem \eqref{App} is in fact bounded and smooth. Moreover, under the additional structural hypotheses \eqref{f2} and \eqref{f3}, we establish a global uniform $L^1$ estimate
on the right-hand side of \eqref{App} along the whole class of solutions. Once the uniform \(L^1\) control on the nonlinearity is available, the uniform $L^\infty$ bound is deduced by means of different techniques depending on the subcritical or critical behaviour of the nonlinearity. Under the subcritical growth condition
\eqref{f1'}, the \emph{a priori} estimate in \hyperref[a priori estimate2 final]{Theorem \ref{a priori estimate2 final}} is derived by combining the Brezis-Merle inequality
with Adams' inequality and Orlicz-space techniques. Under the purely
exponential growth condition \eqref{f1}, the same conclusion is obtained through a blow-up
analysis in which \hyperref[B-M]{Theorem \ref{B-M}} plays a crucial role, again producing the uniform bound
independent of the solution expressed by \hyperref[a priori estimate1 final]{Theorem \ref{a priori estimate1 final}}.

\medskip

It is well known that a priori $L^\infty$ bounds are a crucial tool for proving the existence of solutions by means of topological degree theory. This is of particular interest for problem \eqref{App}, which is not necessarily variational under the set of assumptions \eqref{f1'} or \eqref{f1}, \eqref{f2} and \eqref{f3}. As a direct consequence of \hyperref[a priori estimate2 final]{Theorem \ref{a priori estimate2 final}} and \hyperref[a priori estimate1 final]{Theorem \ref{a priori estimate1 final}},  we obtain the following existence result.

\begin{theorem}\label{existence of sol}
Assume either \eqref{f1'} or \eqref{f1}, and assume that  \eqref{f2} and \eqref{f3} are satisfied. In addition, suppose that
there exist \(\varepsilon_0>0\) and $\eta \in (0,6)$ such that
\[
    f(x,t)\le \eta t
    \qquad\forall x \in B, \: 0\leq t\le \varepsilon_0.
\tag{\(f_4\)}\label{f5}
\]
Then \eqref{App} has at least one weak solution $u \in C^{3, \gamma}(\overline B)$, for some $0<\gamma<1$.
\end{theorem}

In line with a well-established approach in the literature (see, for instance, \cite{dFLN1982}, \cite{joaorufdjairo2}, and \cite{Mancini}), we prove the existence of solutions by a fixed-point index argument in a cone, applying Krasnosel'ski\u{\i}'s theorem in the form presented by de Figueiredo in \cite[Theorem 3.1]{Djairolecturenotes}. The additional assumption \eqref{f5}  essentially governs the growth of the nonlinear term $f(x,t)$ for small $t$. The restriction $\eta \in (0,6)$ in \eqref{f5} will become clear from the estimate \eqref{InverseBiharmonic} for the inverse biharmonic operator on $C(\overline B)$.

\subsection{Outline of the paper}

The paper is organized as follows. In \hyperref[green estimates sec 2]{Section \ref{green estimates sec 2}}, we collect the main properties of the biharmonic Green function in the unit ball under Dirichlet boundary conditions, including pointwise bounds and estimates for its derivatives. These properties are then used in \hyperref[BM proof sec 3]{Section \ref{BM proof sec 3}} to prove \hyperref[B-M]{Theorem \ref{B-M}}, namely a Brezis-Merle type inequality for the linear problem \eqref{P}, and its sharpness (see \hyperref[rmk:BMsharp]{Remark \ref{rmk:BMsharp}}).

In \hyperref[app sec 4]{Section \ref{app sec 4}}, we establish two preliminary results for the semilinear Dirichlet biharmonic problem \eqref{App}. First, we prove a regularity result showing that, under an exponential upper growth assumption on $f$, every very weak solution is bounded and smooth; see \hyperref[regularity1]{Theorem \ref{regularity1}}. Second, we obtain a uniform $L^1$ bound on the nonlinearity along the whole solution set; see \hyperref[uniform L11]{Theorem \ref{uniform L11}}.

\hyperref[a priori subcrit]{Section \ref{a priori subcrit}}, \hyperref[subsec a priori crit]{Section \ref{subsec a priori crit}}, and \hyperref[existence]{Section \ref{existence}} are devoted to the proofs of \hyperref[a priori estimate2 final]{Theorem \ref{a priori estimate2 final}}, \hyperref[a priori estimate1 final]{Theorem \ref{a priori estimate1 final}}, and \hyperref[existence of sol]{Theorem \ref{existence of sol}}, respectively.

In \hyperref[system generalization]{Section \ref{system generalization}}, we extend the scalar strategy to non-variational cooperative systems. We discuss the corresponding regularity result, moving-plane argument, uniform $L^1$ bounds, subcritical and critical a priori estimates, and existence result.

Finally, \hyperref[Example]{Section \ref{Example}} contains an example of a nonlinearity not covered by the asymptotic separation framework of \cite{Mancini}.

\section{Green function estimates in the unit ball}\label{green estimates sec 2}

Following Boggio \cite{Boggioclassical} (see also \cite[(2.65)]{Gazzola}), the Green function for the biharmonic operator on the unit ball $B\coloneqq B(0,1)$ is the positive map
$G:\overline{B}\times\overline{B}\to\mathds{R}\cup\{+\infty\}$ given by
\begin{equation}\label{Green}
    G(x,y)=\frac{1}{8\pi^2}\int_{1}^{\frac{[XY]}{|y-x|}} (v^{-1}-v^{-3})\,\mathrm{d}v,
\end{equation}
where
$$[XY]\coloneqq \sqrt{|x|^2|y|^2-2\langle x,y\rangle +1}, \quad \forall\, x, y \in \overline B .$$
It is useful to point out that
\begin{equation} \label{Green3}
[XY]^2-|y-x|^2=(1-|x|^2)(1-|y|^2)>0, \quad \forall\, x, y \in B .
\end{equation}
Evaluating the integral in \eqref{Green}, one obtains the closed form, valid for $x\neq y$,
\begin{equation}\label{Green2}
    G(x,y)=\frac{1}{8\pi^2}\left[\log\left(\frac{[XY]}{|y-x|}\right)-\frac{1}{2}\left(1-\frac{|y-x|^2}{[XY]^2}\right)\right].
\end{equation}
In particular, if $y\in\partial B$ then $|y|=1$; hence $[XY]^2=|x-y|^2$, and therefore $G(x,y)=0$. 

In the next lemma we collect the successive derivatives of $G(x,y)$, which can be obtained by direct calculation.
\begin{lemma}\label{sucessive} For every $x$, $y \in \overline{B} $ with $x\neq y$, we have:
    \begin{itemize}
        \item $\nabla_y G(x,y)=\frac{1}{8\pi^2}\left[\frac{a(x,y)}{[XY]^2}-\frac{y-x}{|y-x|^2}+\frac{y-x}{[XY]^2}-\frac{|y-x|^2a(x,y)}{[XY]^4}\right],$
        \item $\Delta_y G(x,y)=\frac{1}{8\pi^2}\left[-\frac{2}{|x-y|^2}+\frac{2(|x|^2+2)}{[XY]^2}-\frac{4\langle a(x,y),y-x \rangle}{[XY]^4}\right],$
        \item $\nabla_y \Delta_y G(x,y)=\frac{1}{2\pi^2}\left[\frac{y-x}{|y-x|^4}-\frac{a(x,y)(|x|^2+3)}{[XY]^4}-\frac{|x|^2(y-x)}{[XY]^4}+\frac{4a(x,y)\langle a(x,y),y-x\rangle}{[XY]^{6}}\right],$
        \item $\Delta^2_y G(x,y)=0,$
    \end{itemize}
where
$$a(x,y)\coloneqq |x|^2y-x.$$
\end{lemma}
With explicit formulas for all the derivatives at hand, we can derive the following estimates.
\begin{lemma}\label{estimates}For every $x$, $y \in \overline B $ with $x \neq y$, we have:
\begin{itemize}
    \item $G(x,y)\leq \frac{1}{8\pi^2}\log\left(\frac{2}{|x-y|}\right),$
    \item $|\nabla_y G(x,y)|\leq\frac{1}{2\pi^2}\frac{1}{|x-y|},$
    \item $|\Delta_y G(x,y)|\leq \frac{3}{2\pi^2}\frac{1}{|x-y|^2}. $
\end{itemize}
\end{lemma}
\begin{proof} Since $[XY]^2-|y-x|^2\ge 0$, we have $|y-x|^2\le [XY]^2$ and therefore
\[
-\frac12\left(1-\frac{|y-x|^2}{[XY]^2}\right)\le 0.
\]
Combining this with \eqref{Green2} yields
\[
G(x,y)\le \frac{1}{8\pi^2}\log\!\left(\frac{[XY]}{|y-x|}\right).
\]
Next, by the Cauchy--Schwarz inequality,
$$
[XY]^2
=1+|x|^2|y|^2-2\langle x,y\rangle
\le 1+|x|^2|y|^2+2|x||y|
=(1+|x||y|)^2
\le 4,
$$
and hence
\[
G(x,y)\le \frac{1}{8\pi^2}\log\!\left(\frac{2}{|y-x|}\right).
\]

To estimate the gradient of $G$, we use the formulas provided by \hyperref[sucessive]{Lemma \ref{sucessive}} and the inequalities $[XY] \geq |y-x|$ and $|a(x,y)|=|x|[XY]\leq [XY]$, which are valid in $\overline{B}$:
\begin{align*}
    |\nabla_y G(x,y)|&\leq \frac{1}{8\pi^2}\left(\frac{|a(x,y)|}{[XY]^2}+\frac{|y-x|}{|y-x|^2}+\frac{|y-x|}{[XY]^2}+\frac{|y-x|^2|a(x,y)|}{[XY]^4}\right)\\
    &=\frac{1}{8\pi^2}\left(|a(x,y)|\left(\frac{1}{[XY]^2}+\frac{|y-x|^2}{[XY]^4}\right)+|y-x|\left(\frac{1}{|y-x|^2}+\frac{1}{[XY]^2}\right)\right)\\
    &\leq \frac{1}{8\pi^2}\left(2\frac{|a(x,y)|}{[XY]^2}+2\frac{|y-x|}{|y-x|^2}\right)
    \leq \frac{1}{4\pi^2}\left(\frac{1}{[XY]}+\frac{1}{|x-y|}\right)
    \leq\frac{1}{2\pi^2}\frac{1}{|y-x|}.
\end{align*}

From \hyperref[sucessive]{Lemma \ref{sucessive}}, we have
\begin{equation*}
    |\Delta_y G(x,y)|\leq \frac{1}{8\pi^2}\left[\frac{2}{|x-y|^2}+\frac{2(|x|^2+2)}{[XY]^2}+\frac{|4\langle a(x,y),y-x\rangle|}{[XY]^4}\right].
\end{equation*}
Finally, using again the identity $|a(x,y)|=|x|[XY]$, the inequalities $[XY] \geq |x-y|$ and $|x|<1$ in $\overline{B}$, and the Cauchy-Schwarz inequality. we deduce that
\begin{equation*}
    |\Delta_yG(x,y)|\leq \frac{1}{8\pi^2}\left(\frac{2}{|x-y|^2}+\frac{6}{|x-y|^2}+\frac{4}{|x-y|^2}\right)=\frac{3}{2\pi^2}\frac{1}{|x-y|^2}.
\end{equation*}
\end{proof}
The next lemma concerns the distributional identity satisfied by the Green function $G(x,y)$. In its proof we will use the well-known Green identity:
\begin{equation}\label{GI}
\int_{B}\bigl(u\,\Delta v - v\,\Delta u\bigr)\,\mathrm{d}x
=
\int_{\partial B}\left(u\,\frac{\partial v}{\partial \nu}-v\,\frac{\partial u}{\partial \nu}\right)\,\mathrm{d}S,
\qquad u,v\in C^{2}(\overline{B}),
\end{equation}
where $\nu$ denotes the outward unit normal vector field on $\partial B$ and
\[
\frac{\partial u}{\partial \nu}(x)= \partial_{\nu}u(x) \coloneqq \langle\nabla u(x), \nu(x)\rangle,
\qquad x\in\partial B.
\]
\begin{lemma}\label{dist property} For every $\varphi \in C_{\mathrm{cl}}^{4}(\overline{B})$ and $x \in B$, it holds
\begin{equation}\label{distribution property}
    \varphi(x)=\int_{B} G(x,y)\Delta^2\varphi(y)\,\mathrm{d}y.
\end{equation}  
\end{lemma}
\begin{proof} Let $\varphi \in C_{\mathrm{cl}}^{4}(\overline{B})$. Fix $x\in B$ and let $\varepsilon>0$ be such that $B(x,\varepsilon)\subset \subset B$. We split
$$
\int_{B} G(x,y)\,\Delta^{2}\varphi(y)\,\mathrm{d}y
=\int_{B(x,\varepsilon)} G(x,y)\,\Delta^{2}\varphi(y)\,\mathrm{d}y
+\int_{B\setminus B(x,\varepsilon)} G(x,y)\,\Delta^{2}\varphi(y)\,\mathrm{d}y
\eqqcolon I_{\varepsilon}+J_{\varepsilon}.
$$
By \hyperref[estimates]{Lemma~\ref{estimates}}, the fact that $G(x,y)\geq 0$, and the elementary inequality $\log t\le t$ for all $t>0$, we obtain
\begin{align*}
|I_{\varepsilon}|
&=\left|\int_{B(x,\varepsilon)} G(x,y)\,\Delta^{2}\varphi(y)\,\mathrm{d}y\right|
\le \|\Delta^{2}\varphi\|_{L^{\infty}(B)}\,\frac{1}{8\pi^{2}}
\int_{B(x,\varepsilon)} \log\!\left(\frac{2}{|x-y|}\right)\,\mathrm{d}y\\
&\le \|\Delta^{2}\varphi\|_{L^{\infty}(B)}\,\frac{1}{4\pi^{2}}
\int_{B(x,\varepsilon)} \frac{1}{|x-y|}\,\mathrm{d}y
=\|\Delta^{2}\varphi\|_{L^{\infty}(B)}\,\frac{1}{4\pi^{2}}
\left(\frac{2\pi^{2}}{3}\varepsilon^{3}\right).
\end{align*}

Now, we apply Green's identity \eqref{GI} to $J_{\varepsilon}$; since for any fixed $x\in B$
\[
G(x,y)=\frac{\partial G}{\partial \nu_{y}}(x,y)=0
\quad\forall\,y\in\partial B,
\]
the corresponding boundary integrals over $\partial B$ vanish, and we obtain
\begin{align*}
J_{\varepsilon}
&=\int_{B\setminus B(x,\varepsilon)} G(x,y)\,\Delta^{2}\varphi(y)\,\mathrm{d}y\\
&=\int_{B\setminus B(x,\varepsilon)} \Delta_y G(x,y)\,\Delta\varphi(y)\,\mathrm{d}y
-\int_{\partial B(x,\varepsilon)} G(x,y)\,\frac{\partial(\Delta\varphi)}{\partial \nu}(y)\,\mathrm{d}S_y
+\int_{\partial B(x,\varepsilon)} \Delta\varphi(y)\,\frac{\partial G}{\partial \nu_{y}}(x,y)\,\mathrm{d}S_y\\
&\eqqcolon K_{\varepsilon}-J_{\varepsilon}'+J_{\varepsilon}'' .
\end{align*}
We estimate $J_{\varepsilon}'$ exactly as we estimated $I_{\varepsilon}$ and obtain a constant $C>0$ such that
\[
|J_{\varepsilon}'|
=\left|\int_{\partial B(x,\varepsilon)} G(x,y)\,\frac{\partial(\Delta\varphi)}{\partial \nu}(y)\,\mathrm{d}S_y\right|
\le C\,\varepsilon^{2}.
\]
To estimate $J_{\varepsilon}''$, we apply the gradient estimate in \hyperref[estimates]{Lemma \ref{estimates}} to obtain
\begin{align*}
|J_{\varepsilon}''|
&=\left|\int_{\partial B(x,\varepsilon)} \Delta\varphi(y)\,\frac{\partial G}{\partial \nu_{y}}(x,y)\,\mathrm{d}S_y\right|
=\left|\int_{\partial B(x,\varepsilon)} \Delta\varphi(y)\,\langle\nabla_y G(x,y),\nu(y)\rangle\,\mathrm{d}S_y\right|\\
&\le \|\Delta\varphi\|_{L^{\infty}(B)}\int_{\partial B(x,\varepsilon)} |\nabla_y G(x,y)|\,\mathrm{d}S_y
\leq\|\Delta\varphi\|_{L^{\infty}(B)}\frac{1}{2\pi^2}\int_{\partial B(x,\varepsilon)}\frac{1}{|y-x|}\,\mathrm{d}S_y
= \varepsilon^2\lVert \Delta \varphi \rVert_{L^{\infty}(B)}.
\end{align*}
Next, we apply \eqref{GI} again to $K_{\varepsilon}$; since $\varphi \in C_{\text{cl}}^4(\overline{B})$, the boundary integrals over $\partial B$ vanish, and we obtain
\begin{align*}
    K_{\varepsilon}&=\int_{B\setminus B(x,\varepsilon)} \Delta_y G(x,y)\Delta \varphi(y)\,\mathrm{d}y\\
    &=\int_{B\setminus B(x,\varepsilon)} \Delta^2_y G(x,y)\varphi(y)\,\mathrm{d}y-\int_{\partial B(x,\varepsilon)} \Delta_yG(x,y)\frac{\partial \varphi}{\partial \nu}(y)\,\mathrm{d}S_y+\int_{\partial B(x,\varepsilon)} \varphi(y)\frac{\partial (\Delta_y G)}{\partial \nu}(x,y)\,\mathrm{d}S_y.
    \end{align*}
Moreover, by  \hyperref[sucessive]{Lemma \ref{sucessive}}, we have $\Delta^2_yG(x,y)=0$ for every $y \in B \setminus B(x, \varepsilon)$, and hence
$$
    K_{\varepsilon}=-\int_{\partial B(x,\varepsilon)} \Delta_yG(x,y)\frac{\partial \varphi}{\partial \nu}(y)\,\mathrm{d}S_y+\int_{\partial B(x,\varepsilon)} \varphi(y)\frac{\partial (\Delta_y G)}{\partial \nu}(x,y)\,\mathrm{d}S_y\\
    =-K_{\varepsilon}'+K_{\varepsilon}'',
$$
From \hyperref[estimates]{Lemma \ref{estimates}}, we find
$$
    |K_{\varepsilon}'|=\left|\int_{\partial B(x,\varepsilon)}\Delta_y G(x,y)\frac{\partial \varphi}{\partial \nu}(y)\,\mathrm{d}S_y\right|
    \leq \text{$\frac{3 \lVert \nabla \varphi \rVert_{L^\infty(B)}}{2 \pi^2}$} \int_{\partial B(x,\varepsilon)}\frac{1}{|x-y|^2}\,\mathrm{d}S_y
     \text{=$3 \lVert \nabla \varphi \rVert_{L^\infty(B)}$}\varepsilon.
$$
In order to estimate $|K_{\varepsilon}''|$, we apply \hyperref[sucessive]{Lemma \ref{sucessive}}, use that $\nu(y)=(y-x)/|y-x|$, and write
\begin{align*}
    K_{\varepsilon}''&=\int_{\partial B(x,\varepsilon)}\varphi(y)\frac{\partial(\Delta_yG)}{\partial \nu_{y}}(x,y)\,\mathrm{d}S_y
    =\int_{\partial B(x,\varepsilon)} \varphi(y) \langle\nabla_y \Delta_y G(x,y), \nu(y)\rangle\,\mathrm{d}S_y\\
    &=\frac{1}{2\pi^2\varepsilon^3}\int_{\partial B(x,\varepsilon)}\varphi(y)\,\mathrm{d}S_y-\frac{1}{2\pi^2\varepsilon}\int_{\partial B(x,\varepsilon)}\varphi(y)\frac{|x|^2+3}{[XY]^4}\langle a(x,y),y-x\rangle\,\mathrm{d}S_y\\
    &\hspace{0.1cm}-\frac{\varepsilon}{2\pi^2}\int_{\partial B(x,\varepsilon)}\varphi(y)\frac{|x|^2}{[XY]^4}\,\mathrm{d}S_y+\frac{2}{\pi^2\varepsilon}\int_{\partial B(x,\varepsilon)}\varphi(y)\frac{\langle a(x,y),y-x\rangle^2}{[XY]^6}\mathrm{d}S_y\\
    &=K_{\varepsilon,1}''-K_{\varepsilon,2}''-K_{\varepsilon,3}''+K_{\varepsilon,4}''.
\end{align*}
We have
$$
    K_{\varepsilon,1}''=\frac{1}{2\pi^2\varepsilon^3}\int_{\partial B(x,\varepsilon)} \varphi(y)\,\mathrm{d}S_y
    =\frac{1}{|\partial B(x,\varepsilon)|}\int_{\partial B(x,\varepsilon)} \varphi(y)\,\mathrm{d}S_y\to \varphi(x)\quad \text{as}\quad \varepsilon \to 0^{+},
$$
and it remains to analyze the other integrals $K_{\varepsilon,2}''$, $K_{\varepsilon,3}''$, and $K_{\varepsilon,4}''$. To this aim, for every $y \in \partial B(x, \varepsilon)$, we write $y=x+h$ with $|h|=\varepsilon$, then
\begin{align*}
[XY]^2
&=|x|^{2}|x+h|^{2}-2\langle x,x+h\rangle+1
=(1-|x|^{2})^{2}+2\langle x,h\rangle(|x|^{2}-1)+|x|^{2}\varepsilon^{2}\\
&\ge (1-|x|^{2})^{2}-2(1-|x|^{2})\langle x,h\rangle
\ge c_{x}-2\sqrt{c_{x}}|x|\,\varepsilon,
\end{align*}
where $c_{x}\coloneqq (1-|x|^{2})^{2}$. If $x\neq 0$, we choose $\varepsilon>0$ such that
\[
0<\varepsilon<\varepsilon_{0}(x)\coloneqq \frac{c_{x}}{4\sqrt{c_{x}}|x|}
=\frac{1-|x|^{2}}{4|x|}.
\]
In this way
\[
[XY]^2\ge c_{x}-\frac{c_{x}}{2}=\frac{c_{x}}{2}.
\]
Observe that if $x=0$, then $[XY]^2=1$, and the same bound $[XY]^2\ge c_{x}/2$ holds as well. Consequently, for every $\varepsilon\le \varepsilon_{0}$ and every $y \in \partial B(x, \varepsilon)$,
\[
\frac{1}{[XY]^2}\le \frac{2}{c_{x}}=\frac{2}{(1-|x|^{2})^{2}}.
\]
From the inequality above together with $|a(x,y)|=|x|[XY] \leq [XY]$, we can estimate the remaining integrals. Indeed, for every $\varepsilon<\varepsilon_0$, we have
\begin{align*}
    |K_{\varepsilon,2}''|&\leq\frac{1}{2\pi^2\varepsilon}\int_{\partial B(x,\varepsilon)}\left|\varphi(y)\right|\frac{|x|^2+3}{[XY]^4}|\langle a(x,y),y-x \rangle|\,\mathrm{d}S_y\\
    &\leq \frac{4}{2\pi^2\varepsilon} \lVert \varphi \rVert_{L^{\infty}(B)} \left(\frac{2}{c_x} \right)^{3/2}\int_{\partial B(x,\varepsilon)}|y-x|\,\mathrm{d}S_y
    =4\varepsilon^3\lVert \varphi\rVert_{L^{\infty}(B)}\left(\frac{2}{c_x} \right)^{3/2},
\end{align*}
moreover
$$
    |K_{\varepsilon,3}''|\leq \frac{\varepsilon}{2\pi^2}\int_{\partial B(x, \varepsilon)} |\varphi(y)|\frac{|x|^2}{[XY]^4}\,\mathrm{d}S_y
    \leq \frac{\varepsilon}{2\pi^2}\lVert \varphi \rVert_{L^{\infty}(B)}\frac{4}{c_x^2}2\pi^2\varepsilon^3
    =\frac{4\lVert \varphi \rVert_{L^{\infty}(B)}}{c_x^2}\varepsilon^4,
$$
and
\begin{align*}
    |K_{\varepsilon,4}''|&\leq \frac{2}{\pi^2\varepsilon}\int_{\partial B(x,\varepsilon)}|\varphi(y)|\frac{|\langle a(x,y),y-x\rangle|^2}{[XY]^{6}}\,\mathrm{d}S_y
    \leq \frac{2}{\pi^2\varepsilon}\lVert \varphi \rVert_{L^{\infty}(B)}\int_{\partial B(x,\varepsilon)}\frac{[XY]^2|y-x|^2}{[XY]^6}\,\mathrm{d}S_y\\
    &\leq \frac{2}{\pi^2\varepsilon}\lVert \varphi \rVert_{L^{\infty}(B)}\frac{4}{c_x^2} \int_{\partial B(x,\varepsilon)}|y-x|^2\,\mathrm{d}S_y
    =\frac{16\lVert \varphi \rVert_{L^{\infty}(B)}}{c_x^2}\varepsilon^{4}.
\end{align*}
Thus, 
\begin{equation*}
    \lim_{\varepsilon\to0^{+}}K_{\varepsilon}''=\lim_{\varepsilon\to0^{+}}(K_{\varepsilon,1}''-K_{\varepsilon,2}''-K_{\varepsilon,3}''+K_{\varepsilon,4}'')=\varphi(x).
\end{equation*}

Finally, recalling that $K_{\varepsilon}'$, $J_{\varepsilon}'$, $J_{\varepsilon}''$, $I_{\varepsilon}$ all converge to $0$ as $\varepsilon \to 0^{+}$, we conclude that 
\begin{align*}
    \int_{B}G(x,y)\Delta^2\varphi(y)\,\mathrm{d}y&=\lim_{\varepsilon \to 0^{+}} (I_{\varepsilon}+J_{\varepsilon})
    =\lim_{\varepsilon \to 0^{+}}(I_{\varepsilon}+K_{\varepsilon}-J_{\varepsilon}'+J_{\varepsilon}'')\\
    &=\lim_{\varepsilon \to 0^{+}}(I_{\varepsilon}-K_{\varepsilon}'+K_{\varepsilon}''-J_{\varepsilon}'+J_{\varepsilon}'')
    =\varphi(x),
\end{align*}
and the lemma is proved.
\end{proof}

The preceding distributional identity allows us to construct solutions of the linear problem \eqref{P} with $L^1$ data through the Green representation formula. We record this fact in the following lemma.

\begin{lemma}\label{solution} If $f \in L^1(B)$, then there exists a unique very weak solution $u$ of \eqref{P} in the sense \eqref{very weak solution}, and the following representation formula holds
\begin{equation}\label{Green representation}
    u(x)=\int_{B} G(x,y)f(y)\,\mathrm{d}y.
\end{equation}
\end{lemma}
\begin{proof} Define $u$ as in \eqref{Green representation}. Our first objective is to prove that $u \in L^1(B)$. To achieve this goal, observe that by  the classical Hardy-Littlewood rearrangement inequality (see \cite[Theorem 3.4]{Lieb}), we have
\begin{equation} \label{eq:Hardy-Littlewood}
\int_{B} \frac{1}{|x-y|}\,\mathrm{d}x \leq \int_{B} \frac{1}{|x|}\,\mathrm{d}x, \quad  \forall y \in B.
\end{equation}
Combining \hyperref[estimates]{Lemma \ref{estimates}}, the elementary inequality $\log(2t) \leq t$ for all $t>0$, and \eqref{eq:Hardy-Littlewood}, we get
\begin{align} 
    \int_{B} \int_{B} G(x,y) |f(y)| &\,\mathrm{d}x\,\mathrm{d}y  \leq \frac{1}{8\pi^2}\int_{B}|f(y)|\log\left(\frac{2}{|x-y|}\right)\,\mathrm{d}x\,\mathrm{d}y
    \leq \frac{1}{8\pi^2}\int_{B} \int_{B}\frac{|f(y)|}{|x-y|}\,\mathrm{d}x\,\mathrm{d}y\\ \label{RepFormulaFT}
    &\leq \frac{1}{8\pi^2} \int_{B} |f(y)| \left(\int_{B}\frac{1}{|x|}\,\mathrm{d}x \right)\,\mathrm{d}y
    =\frac{1}{8\pi^2} \int_{B} |f(y)| \left( 2 \pi^2\int_{0}^{1}  s^2\,\mathrm{d}s \right)\,\mathrm{d}y\\
    &= \frac{1}{8\pi^2}\frac{2\pi^2}{3}\lVert f \rVert_{L^1(B)} <+\infty.
\end{align}
Thus, by the  Fubini--Tonelli Theorem, it follows that $u \in L^1(B)$.

Next, we show that $u$ is a very weak solution of \eqref{P}, namely 
\begin{equation*}
    \int_{B}u\Delta^2\varphi\,\mathrm{d}x=\int_{B} f\varphi\,\mathrm{d}x,\quad \forall\varphi \in C_{\mathrm{cl}}^{4}(\overline{B}),
\end{equation*}
or explicitly
\begin{equation*}
\int_{B}\int_{B}G(x,y)f(y)\Delta^2\varphi(x)\,\mathrm{d}y\mathrm{d}x=\int_{B}f\varphi\,\mathrm{d}x,\quad \forall\varphi \in C_{\mathrm{cl}}^{4}(\overline{B}).
\end{equation*}
Let $\varphi \in C_{\mathrm{cl}}^{4}(\overline{B})$. Arguing as in \eqref{RepFormulaFT}, we get
$$
    \int_{B} \int_{B} G(x,y) |f(y)|\,|\Delta^2 \varphi(x)|\,\mathrm{d}x\,\mathrm{d}y \leq \lVert\Delta^2 \varphi\lVert_{L^{\infty}(B)}\int_{B} \int_{B} G(x,y) |f(y)|\,\,\mathrm{d}x\,\mathrm{d}y <+\infty
$$
Thus, again by Fubini-Tonelli Theorem and \hyperref[dist property]{Lemma \ref{dist property}}, we conclude that

\begin{align*}
    \int_{B} \int_{B} & G(x,y)  f(y) \Delta^2 \varphi(x)\,\mathrm{d}y\,\mathrm{d}x=\int_{B} \int_{B} G(x,y)  f(y) \Delta^2 \varphi(x)\,\mathrm{d}x\,\mathrm{d}y\\
    &= \int_{B}  f(y)  \int_{B} G(x,y) \Delta^2 \varphi(x)\,\mathrm{d}x\,\mathrm{d}y
    =\int_{B} f (y)  \varphi(y)\,\mathrm{d}y
    =\int_{B} f(x)\varphi(x)\,\mathrm{d}x.
\end{align*}

Finally, to prove uniqueness, suppose that there exists another solution $v \in L^1(B)$ of \eqref{P}. This implies that
\begin{equation*}
    \int_{B} (u-v)\Delta^2\varphi\,\mathrm{d}x=0,\quad \forall \varphi \in C_{\mathrm{cl}}^4(\overline{B}).
    \end{equation*}
Proceeding as in the beginning of the proof of  \cite[Lemma 1]{BrezisCazenave}, we conclude that $u=v$.
\end{proof}

\section{Proof of the Brezis-Merle type inequality}\label{BM proof sec 3}

This section is devoted to the proof of {\hyperref[B-M]{Theorem \ref{B-M}}} and some of its direct consequences.

\begin{proof}[Proof of {\hyperref[B-M]{Theorem \ref{B-M}}}]By \hyperref[solution]{Lemma \ref{solution}}, the unique solution of \eqref{P} in the sense \eqref{very weak solution} is given by
\begin{equation*}
    u(x)=\int_{B}G(x,y)f(y)\,\mathrm{d}y.
\end{equation*}
Let $C>0$. By \hyperref[estimates]{Lemma \ref{estimates}}, we have
\begin{align*}
    \int_{B}\exp\left[\frac{C|u(x)|}{\lVert f\rVert_{L^1(B)}}\right]\,\mathrm{d}x&\leq \int_{B} \exp\left[\frac{C}{\lVert f\rVert_{L^1(B)}}\int_{B}|G(x,y)||f(y)|\,\mathrm{d}y\right]\mathrm{d}x\\
    &\leq \int_{B} \exp\left[\frac{C}{8\pi^2\lVert f\rVert_{L^1(B)}}\int_{B}\log\left(\frac{2}{|x-y|}\right)|f(y)|\,\mathrm{d}y\right]\mathrm{d}x\\
    &=\int_{B}\exp\left[\int_{B}\log\left(\frac{2}{|x-y|}\right)^{\frac{C}{8\pi^2}}\,\mathrm{d}\mu\right]\,\mathrm{d}x,
\end{align*}
where $\mu$ is the probability measure in $B$ defined by $mathrm{d}\mu=\frac{|f(y)|}{\lVert f\rVert_{L^{1}(B)}}\,\mathrm{d}y$. Since $\exp$ is a convex function in $\mathds{R}$, we can apply Jensen's inequality in order to obtain
\begin{align*}
    \int_{B}\exp\left[\frac{C|u(x)|}{\lVert f\rVert_{L^1(B)}}\right]\,\mathrm{d}x&\leq\int_{B}\exp\left[\int_{B}\log\left(\frac{2}{|x-y|}\right)^{\frac{C}{8\pi^2}}\,\mathrm{d}\mu\right]\,\mathrm{d}x
    \leq \int_{B}\int_{B}\left(\frac{2}{|x-y|}\right)^{\frac{C}{8\pi^2}}\frac{|f(y)|}{\lVert f\rVert_{L^1(B)}}\,\mathrm{d}y\,\mathrm{d}x.
\end{align*}
Now with the objective of applying the Fubini-Tonelli Theorem, we use the classical Hardy-Littlewood rearrangement  inequality (see \cite[Theorem 3.4]{Lieb}) and, similarly to \eqref{eq:Hardy-Littlewood}, we can estimate
\begin{align*}
    \int_{B}  \int_{B} &\left(\frac{2}{|x-y|}\right)^{\frac{C}{8\pi^2}}\frac{|f(y)|}{\lVert f \rVert_{L^1(B)}}\,\mathrm{d}x\,\mathrm{d}y = 2^{\frac{C}{8\pi^2}} \int_{B} \frac{|f(y)|}{\lVert f \rVert_{L^1(B)}} \left( \int_{B}\frac{1}{|x-y|^{\frac{C}{8\pi^2}}}\,\mathrm{d}x \right)\,\mathrm{d}y\\
    &\leq 2^{\frac{C}{8\pi^2}} \int_{B} \frac{|f(y)|}{\lVert f \rVert_{L^1(B)}} \left(\int_{B}\frac{1}{|x|^{\frac{C}{8\pi^2}}}\,\mathrm{d}x \right)\,\mathrm{d}y
    = 2^{\frac{C}{8\pi^2}}\int_{B} \frac{|f(y)|}{\lVert f \rVert_{L^1(B)}} \left(2\pi^2\int_{0}^{1}s^{3-\frac{C}{8\pi^2}}\,\mathrm{d}s \right)\,\mathrm{d}y.\\
\end{align*}
We can now choose $C>0$ such that $3-C/(8\pi^2)>-1$, that is $0<C<32\pi^2$,
and this is enough to conclude that
$$
    \int_{B} \exp\left[\frac{C\lvert u(x) \rvert}{\lVert f \rVert_{L^1(B)}}\right]\,\mathrm{d}x 
    \leq  2^{\frac{C}{8\pi^2}}\frac{2\pi^2}{4-\frac{C}{8\pi^2}}=2^{\frac{C}{8\pi^2}}\frac{16\pi^4}{32\pi^2-C}.
$$
Letting $\delta=32\pi^2-C$, the result follows, indeed
\begin{equation*}
    \frac{C}{8\pi^2}=4-\frac{\delta}{8\pi^2}, \quad 0<4-\frac{\delta}{8\pi^2}<4 \quad \text{and}\quad 2^{\frac{C}{8\pi^2}}\frac{16\pi^4}{32\pi^2-C} < \frac{2^8 \pi^4}{\delta}.
\end{equation*}
\end{proof}

\begin{remark} \label{rmk:BMsharp}
The constant $32\pi^2$ in \hyperref[B-M]{Theorem \ref{B-M}} is sharp. To see this, fix $p\in B$ and consider a family of nonnegative functions $f_{\varepsilon}\in L^1(B)$ such that $\lVert f_\varepsilon\rVert_{L^1(B)}=1$ and $f_{\varepsilon}\to \delta_p$ as $\varepsilon\to 0^+$. If $u_\varepsilon$ is the corresponding very weak solution (see \hyperref[solution]{Lemma \ref{solution}}), then, away from $p$, $u_\varepsilon$ converges to the Green function $G(\cdot,p)$. Since, by \eqref{Green2},
\begin{equation*}
G(x,p)=\frac{1}{8\pi^2}\log\frac{1}{|x-p|}+O(1)
\quad\text{as } x\to p,
\end{equation*}
we have, formally,
\begin{equation*}
\exp(\alpha G(x,p))\sim |x-p|^{-\frac{\alpha}{8\pi^2}}.
\end{equation*}
This function is locally integrable in dimension four if and only if $\alpha/(8\pi^2)<4$, that is, if and only if $\alpha<32\pi^2$.
\end{remark}

As a direct consequence of \hyperref[B-M]{Theorem \ref{B-M}}, the
Brezis-Merle inequality can be scaled to arbitrary balls by means of a change of variables.
\begin{corollary}\label{B-M general ball}
Let $x_0\in\mathds{R}^4$ and $R>0$. If $f \in L^1(B(x_0,R))$ and $u \in L^1(B(x_0,R))$ is a very weak solution of
\begin{equation}
\left\{\begin{array}{rcl}
\Delta^2 u  =  f(x) \ \ &\text{in} \ \ B(x_0,R),\\
u  =  \partial_{\nu}u = 0\ \ &\text{on} \ \ \partial B(x_0,R),\\
\end{array}\right.
\end{equation}
then for every $\delta \in (0,32 \pi^2)$ we have
\begin{equation}\label{BM-ball-R}
\int_{B(x_0,R)} 
\exp\!\left(\frac{(32\pi^2-\delta)\,|u(x)|}{\|f\|_{L^1(B(x_0,R))}}\right)\,\mathrm{d}x
\le \frac{256\pi^4}{\delta}R^4.
\end{equation}
\end{corollary}

The next corollary shows that every very weak solution of the linear problem \eqref{P} enjoys exponential integrability of arbitrary finite order.

\begin{corollary}\label{integrability} If $u$ is a very weak solution of \eqref{P}, then $e^{k|u|} \in L^1(B)$ for every $k>0$.
\end{corollary}
\begin{proof} Let $u$ be a very weak solution of \eqref{P} with $f \in L^1(B)$. Fix $k>0$ and let $\varepsilon>0$ be such that $0<\varepsilon<1/k$.
Let $\varphi$ be a simple function such that
\begin{equation*}
    \lVert f-\varphi\rVert_{L^1(B)}<\varepsilon.
\end{equation*}
We observe that $f=f_1+f_2$, where $f_1\coloneq f-\varphi$ and $f_2\coloneqq \varphi$; in particular, $\lVert f_1\lVert_{L^1(B)}<\varepsilon$ and $\lVert f_2\rVert_{L^{\infty}(B)}<+\infty$. By \hyperref[solution]{Lemma \ref{solution}}, there exist unique solutions to the following problems with $i=1,2$:
\begin{equation} \label{eq:LP12}
\left\{\begin{array}{rcl}
\Delta^2 u_i  =  f_i \ \ &\text{in} \ \ B,\\
u_i  =  \partial_{\nu}u_i = 0\ \ &\text{on} \ \ \partial B.\\
\end{array}\right.
\end{equation}
By the uniqueness of solutions for $\eqref{P}$, we deduce that $u=u_1+u_2$. Since $u_1$ is a very weak solution of \eqref{eq:LP12} with $i=1$, we can apply the Brezis-Merle inequality with $\delta=32\pi^2-1$ and obtain that
\begin{equation*}
    \int_{B}e^{k|u_1(x)|}\,\mathrm{d}x\leq \int_{B}e^{\frac{|u_1(x)|}{\varepsilon}}\,\mathrm{d}x\leq \int_{B}e^{\frac{|u_1(x)|}{\lVert f_1\rVert_{L^1(B)}}}\,\mathrm{d}x<+\infty.
\end{equation*}
Now, since $f_2 \in L^{\infty}(B)$, it follows from \hyperref[estimates]{Lemma \ref{estimates}}, the elementary inequality $\log(t) \leq t$ for all $t>0$, and \eqref{eq:Hardy-Littlewood} that
\begin{align*}
    |u_2(x)|&\leq \int_{B} |G(x,y)||f_2(y)|\,\mathrm{d}y
    \leq \frac{\lVert f_2\rVert_{L^{\infty}(B)}}{4\pi^2}\int_{B} \frac{1}{|x-y|}\,\mathrm{d}y
    \leq \frac{\lVert f_2\rVert_{L^{\infty}(B)}}{4\pi^2} \int_{B}\frac{1}{|y|}\,\mathrm{d}y
    =\frac{\lVert f_2\rVert_{L^{\infty}(B)}}{6},
\end{align*}
that is, $u_2 \in L^{\infty}(B)$. Finally,
$$
    \int_{B} e^{k|u(x)|}\,\mathrm{d}x\leq \int_{B} e^{k|u_1|+k|u_2|}\,\mathrm{d}x
    \leq e^{k\frac{\lVert f_2 \rVert_{L^{\infty}(B)}}{6}}\int_{B}e^{k|u_1(x)|}\,\mathrm{d}x<+\infty.
$$
\end{proof}
\section{Regularity and $L^1$ bound on the nonlinearity  in the scalar case}\label{app sec 4}

In this section, we consider the scalar problem \eqref{App} assuming that the nonlinearity $f: \overline{B} \times \mathds{R}^{+} \to \mathds{R}^{+}$ is continuous, and satisfies the following general growth condition:
\begin{flalign}\label{f0}
     \exists \, c,p>0 \text{ such that } f(x,t)\leq ce^{pt}, \quad \forall x\in B, \; t \geq 0.
\end{flalign}
The growth condition \eqref{f0} is motivated by the Brezis-Merle inequality and its consequences, which we will apply extensively throughout this section. We emphasize that the two relevant growth scenarios, \eqref{f1'} and \eqref{f1}, are both encompassed by the general exponential bound \eqref{f0}.

\subsection{Regularity of very weak solutions}\label{subsec reg}
Our first result concerns the regularity of solutions of \eqref{App} and is a direct application of \hyperref[integrability]{Corollary \ref{integrability}}.
\begin{theorem}\label{regularity1} Assume that $\eqref{f0}$ holds. Then every very weak solution $u$ of \eqref{App} satisfies $u \in C^{3,\gamma}(\overline{B}) \cap L^{\infty}(B)$ for some $0<\gamma<1$.
\end{theorem}
\begin{proof} Let $u \in L^1(B)$ be such that $f(x,u) \in L^1(B)$ and assume that $u$ is  a very weak solution of \eqref{App} in the sense \eqref{very weak solution}. By \hyperref[integrability]{Corollary \ref{integrability}}, it follows that
\begin{equation*}
    \int_{B} e^{ku}\,\mathrm{d}x<+\infty\quad \forall\, k>0.
\end{equation*}
Thus, from $\eqref{f0}$, we get for every $r\geq 1$
\begin{align*}
    \lVert f(x,u)\rVert_{L^r(B)}^{r}=\int_{B} |f(x,u)|^{r}\,\mathrm{d}x\leq c^r\int_{B}e^{rpu(x)}\,\mathrm{d}x<+\infty.
\end{align*}

Now, from $L^p$ regularity theory (see, for instance, \cite[Corollary 2.21]{Gazzola}), there exists a (unique) strong solution $\tilde{u} \in W^{4,r}(B) \cap W_0^{2,r}(B)$ of the problem
 \begin{equation}\label{reg problem}
\left\{\begin{array}{rcl}
\Delta^2 \tilde{u}  =  f(x,u) \ \ &\text{in} \ \ B,\\
\tilde{u}  =  \partial_{\nu}\tilde{u} = 0\ \ &\text{on} \ \ \partial B,\\
\end{array}\right.
\end{equation}
moreover,
\begin{equation}\label{Lp estimate ball reg}
    \lVert \tilde{u}\rVert_{W^{4,r}(B)}\leq C\lVert f(x,u)\rVert_{L^r(B)}.
\end{equation}
In particular $\tilde{u}$ is a solution of \eqref{reg problem}  in the very weak sense  and, as a consequence, the function $w\coloneq u -\tilde{u}$ is a very weak solution of the problem
 \begin{equation*}
\left\{\begin{array}{rcl}
\Delta^2 w  =  0 \ \ &\text{in} \ \ B,\\
w  =  \partial_{\nu}w = 0\ \ &\text{on} \ \ \partial B,\\
\end{array}\right.
\end{equation*}
that is
\begin{equation*}\label{equality of dist and strong sol}
    \int_{B} w \Delta^2\varphi\,\mathrm{d}x=0\quad \forall \varphi \in C_{\mathrm{cl}}^4(\overline{B}).
\end{equation*}

Next, proceeding as in \cite[Lemma 1]{BrezisCazenave}, we have  $w=0$ and, in particular,
\[
u=\tilde{u}\in W^{4,r}(B)\cap W^{2,r}_0(B).
\]
Moreover, in view of \eqref{Lp estimate ball reg}, we arrive at the estimate

\begin{equation*}
    \lVert u\rVert_{W^{4,r}(B)}\leq C \lVert f(x,u)\rVert_{L^{r}(B)}<+\infty\quad \forall r\geq1
\end{equation*}
and we can conclude by Morrey's inequality that $u \in C^{3,\gamma}(\overline{B}) \cap L^{\infty}(B)$ for some $0<\gamma<1.$
\end{proof}

In view of \hyperref[regularity1]{Theorem \ref{regularity1}}, whenever we consider a very weak solution $u$ of  \eqref{App} with $f: \overline{B} \times \mathds{R}^{+} \to \mathds{R}^{+}$ satisfying $\eqref{f0}$, we can assume without loss of generality that $u \in C^{3,\gamma}(\overline{B}) \cap L^{\infty}(B)$, for some $0<\gamma<1$. In particular, \hyperref[regularity1]{Theorem \ref{regularity1}} allows us to work thereafter within the class of weak solutions, while keeping the hypotheses and estimates formulated at the very weak level. Specifically, under the assumptions \eqref{f1'} or \eqref{f1} of
\hyperref[a priori estimate2 final]{Theorem \ref{a priori estimate2 final}} or \hyperref[a priori estimate1 final]{Theorem \ref{a priori estimate1 final}} respectively, every very weak solution of \eqref{App} belongs to $C^{3,\gamma}(\overline{B}) \cap L^{\infty}(B)$, for some $0<\gamma<1$.

The next result shows that the regularity of a very weak solution can be improved to $C^{4,\gamma}(\overline B)$ under the following 
\emph{Hölder--Lipschitz structure condition} on $f$: for every $M>0$,
there exist constants $C_M,L_M>0$ such that for all $x,y\in\overline B$ and all $s,t\in[0,M]$,
\begin{align}
|f(x,t)-f(y,t)| &\le C_M |x-y|^\gamma, \label{eq:Holderx}\\
|f(x,t)-f(x,s)| &\le L_M |t-s|. \label{eq:Lipschitzt}
\end{align}

\begin{corollary}\label{rem:C4alpha}
Assume $\eqref{f0}$ and assume that  $u \in C^{3,\gamma}(\overline{B}) \cap L^{\infty}(B)$, for some $0<\gamma<1$, is a solution of \eqref{App}. In addition, suppose that \eqref{eq:Holderx} and \eqref{eq:Lipschitzt} hold with the same $0< \gamma < 1$. Then  $f(\cdot,u(\cdot))\in C^{0,\gamma}(\overline B)$, and therefore $u\in C^{4,\gamma}(\overline B)$. Hence, in particular $u\in C^4(\overline B)$ and $u$ is a classical solution of \eqref{App}.
\end{corollary}
\smallskip
\begin{proof}
Let $M:=\|u\|_{L^\infty(B)}$ and set $h(x):=f(x,u(x))$ for $x\in\overline B$.
For any $x,y\in\overline B$, we have
\[
|h(x)-h(y)|
=
|f(x,u(x))-f(y,u(y))|
\le
|f(x,u(x))-f(y,u(x))|
+
|f(y,u(x))-f(y,u(y))|.
\]
The first term is controlled by \eqref{eq:Holderx} (with $t=u(x)\in[0,M]$), hence
\[
|f(x,u(x))-f(y,u(x))|\le C_M |x-y|^\gamma.
\]
The second term is controlled by \eqref{eq:Lipschitzt} (with $u(x),u(y)\in[0,M]$), giving
\[
|f(y,u(x))-f(y,u(y))|\le L_M |u(x)-u(y)|.
\]
Since $u\in C^{0,\gamma}(\overline B)$, there exists $K>0$ such that
$|u(x)-u(y)|\le K|x-y|^\gamma$ for all $x,y\in\overline B$. Consequently,
\[
|h(x)-h(y)|\le (C_M+L_MK)\,|x-y|^\gamma,
\]
which proves that $h=f(\cdot,u(\cdot))\in C^{0,\gamma}(\overline B)$.

Finally, applying the global Schauder estimates for the Dirichlet biharmonic problem in the ball (see for example a more general version in  \cite[Theorem 2.19]{Gazzola}), we infer that the (strong) solution of
\[
\Delta^2 u=h\quad\text{in }B,\qquad u=\partial_\nu u=0\quad\text{on }\partial B
\]
belongs to $C^{4,\gamma}(\overline B)$. This yields $u\in C^{4,\gamma}(\overline B)$.

\end{proof}

\subsection{Global uniform $L^1$ control on the nonlinearity}\label{subsec4.2}
Our final goal is to obtain uniform $L^\infty$ bounds for the scalar problem \eqref{App}. To this end, our strategy consists in first deriving a uniform $L^1$ bound on the nonlinearity. This intermediate result is established in the present section (see \hyperref[uniform L11]{Theorem \ref{uniform L11}} below). Its proof combines the spectral lower growth condition at infinity \eqref{f2} with a moving plane method near the boundary, under suitable structural monotonicity assumptions on \(f(x,t)\) in a neighborhood of \(\partial B\) (see \eqref{f3}). This step is inspired by the non variational topological approach developed for fourth order problems in \cite{Berchio} and \cite{Mancini}, see also \cite{ReichelWeth}. It crucially relies on the sharp estimates for the biharmonic Green function in balls, which are not directly available in general smooth bounded domains; this is the main reason why we restrict our analysis to the ball setting.

Let $\varphi_1$ be an eigenfunction associated with the first eigenvalue $\lambda_1$ of $(\Delta^2, H_0^{2}(B))$. By the theory developed in \cite[Chapter 3]{Gazzola}, we know that $\lambda_1>0$ and $\varphi_1$ can be chosen to be positive in $B$. In the next lemma, in addition to the continuity of $f$ and \eqref{f0}, we assume \eqref{f2}. It is useful to point out that \eqref{f2} yields the existence of  $a>\lambda_1$ such that $$\liminf_{t\to+\infty} \frac{f(x,t)}{t}\geq a \quad  \text{uniformly in } x \in B$$
and hence, for each $0<\varepsilon<a$, there exists $c>0$ such that
\begin{equation}\label{f2 implication}
    f(x,t)\geq (a-\varepsilon)t-c \quad \forall\, t \geq 0,\, x\in \overline{B}.
\end{equation}

\begin{lemma}\label{L1 uniform with phi1} Assume that $\eqref{f0}$ and $\eqref{f2}$ hold. Then there exists a constant $C>0$ such that
    \begin{equation}\label{L1 estimate with phi1}
        \int_{B} f(x,u)\varphi_1\,\mathrm{d}x\leq C,
    \end{equation}
for every very weak solution $u$ of \eqref{App}.
\end{lemma}
\begin{proof} Let $u$ be a very weak solution of \eqref{App}. By \hyperref[regularity1]{Theorem \ref{regularity1}},  $u$ is in particular a weak solution of \eqref{App}. Thus, choosing $\varphi_1 \in H_0^{2}(B)$ as a test function, we obtain the identity
\begin{equation}\label{weak sol of App}
    \int_{B} f(x,u)\varphi_1\,\mathrm{d}x=\int_{B} \Delta u \Delta \varphi_1\,\mathrm{d}x=\lambda_1 \int_{B}u \varphi_1\,\mathrm{d}x.
\end{equation}
From \eqref{f2}, there exists of $a>\lambda_1$ such that \eqref{f2 implication} holds for some $0<\varepsilon<a-\lambda_1$. Consequently,
\begin{equation}\label{weak sol of App2}
    \int_{B} f(x,u)\varphi_1\,\mathrm{d}x\geq \int_{B}[(a-\varepsilon)u-c]\varphi_1\,\mathrm{d}x=(a-\varepsilon)\int_{B} u\varphi_1\,\mathrm{d}x-c\int_{B}\varphi_1\,\mathrm{d}x.
\end{equation}
Setting $c_1 \coloneq c\int_{B}\varphi_1\,\mathrm{d}x$, we observe that $c_1>0$ since $\varphi_1>0$ in $B$. Then, combining \eqref{weak sol of App} and \eqref{weak sol of App2}, we have
\begin{equation*}
    (a-\varepsilon)\int_{B} u \varphi_1\,\mathrm{d}x-c_1\leq \lambda_1\int_{B} u \varphi_1\,\mathrm{d}x.
\end{equation*}
 Rearranging the terms, we get
\begin{equation*}
    \int_{B}u\varphi_1\,\mathrm{d}x\leq \frac{c_1}{(a-\varepsilon)-\lambda_1},
\end{equation*}
which, in turn, yields
\begin{equation*}
    \int_{B} f(x,u)\varphi_1\,\mathrm{d}x=\lambda_1\int_{B} u\varphi_1\,\mathrm{d}x\leq \frac{c_1\lambda_1}{(a-\varepsilon)-\lambda_1}.
\end{equation*}
\end{proof}

Next, we recall a preliminary result inspired by Berchio, Gazzola, and Weth \cite{Berchio} and Mancini and Romani \cite{Mancini}, see also \cite{ReichelWeth}. It guarantees the existence of a neighborhood of the boundary,
$$
\omega=\omega_r\coloneq \{x \in B:\mathrm{dist}(x,\partial B)<r\} \quad \text{ for some sufficiently small } r>0,
$$
where solutions are decreasing along suitable outer directions. Its proof relies on a moving plane technique that uses pointwise estimates of the Green's function in the ball, rather than the maximum principle arguments typically employed for second-order problems. Here, assumption \eqref{f3} comes into play, and it should be understood as a monotonicity condition
near the boundary. In simple terms, according to \eqref{f3}-(i), the nonlinearity is allowed to grow
with the solution, but by \eqref{f3}-(ii) it is not allowed to grow when the spatial
variable moves outward. This is exactly what is needed in order to
compare a point with its reflection when the moving plane is close to
the boundary.

\begin{lemma}\label{moving 6} Assume that \eqref{f0} and \eqref{f3} hold. Then there exists a neighborhood $\omega$ of $\partial B$ and $\delta>0$ such that  every very weak solution $u$ of \eqref{App} satisfies
$$\langle \nabla u(x),\theta\rangle\leq 0,$$
for all  $x \in \omega$ and for all unitary vectors $\theta \in \mathds{R}^4$ such that $|\theta-\nu(x)|<\delta.$
\end{lemma}
\begin{proof} Condition \eqref{f3} corresponds to assumptions H3a) and H3b) in \cite{Mancini}. In addition, \eqref{f0} and \hyperref[regularity1]{Theorem \ref{regularity1}} ensure the regularity needed to apply the moving-plane argument of \cite[Lemma~2.5]{Mancini}.
\end{proof}

We are now ready to obtain a uniform $L^\infty$ estimate of very weak solutions near the boundary.
\begin{lemma}\label{Linfinity bound near the boundary} Assume that \eqref{f0}, \eqref{f2}, and \eqref{f3} hold. Then there exists a neighborhood $\omega$ of $\partial B$ and a constant $C_1>0$ such that $\lVert u\rVert_{L^{\infty}(\omega)}\leq C_1$ for every very weak solution $u$ of \eqref{App}.
\end{lemma}
\begin{proof} Let $u$ be a very weak solution of \eqref{App}. By \hyperref[moving 6]{Lemma \ref{moving 6}},  $u$ satisfies a monotonicity property on a neighborhood $\omega_r$ of $\partial B$ along suitable outer directions. This allows us to argue as in \cite[Equation (8$'$)]{dFLN1982} and identify
a neighborhood $\omega \subseteq \omega_r$ of $\partial B$ such that, for every $y \in \omega$, we can construct a set $I_{y}\subseteq B$ satisfying the following properties:
\begin{itemize}
    \item[-] $u(x) \geq u(y)$ for every $x \in I_{y},$
    \item[-] $|I_{y}|\geq \sigma>0$ where $\sigma$ does not depend on $y$,
    \item[-] $I_{y} \subseteq B\setminus \omega.$
\end{itemize}

\noindent
Thus, by \hyperref[L1 uniform with phi1]{Lemma \ref{L1 uniform with phi1}} and the positivity of $\lambda_1$ and $\varphi_1$, there exists a constant $C>0$ independent of $u$ such that
\begin{align*}
    C\geq \int_{B} f(x,u)\varphi_1\,\mathrm{d}x=\lambda_1\int_{B}u \varphi_1\,\mathrm{d}x\geq \lambda_1\int_{B \setminus \omega} u\varphi_1\,\mathrm{d}x\geq \lambda_1 \min_{B\setminus\omega}\varphi_1\int_{I_{y} }u(x)\,\mathrm{d}x\geq \lambda_1 \left(\min_{B\setminus \omega}\varphi_1\right) u(y)\sigma,
\end{align*}
for every $y \in \omega$ . This implies the desired uniform $L^{\infty}$ boundedness of every solution $u$ of \eqref{App} in $\omega$. 
\end{proof}

We finally deduce our desired uniform $L^1$ estimate on the nonlinearity.

\begin{theorem}\label{uniform L11} Assume that $\eqref{f0},$ $\eqref{f2}$, and $\eqref{f3}$ hold. Then there exists a constant $M>0$ such that
\begin{equation*}
    \int_{B} f(x,u)\,\mathrm{d}x\leq M,
\end{equation*}
for every very weak solution $u$ of \eqref{App}.    
\end{theorem}
\begin{proof} Let $\omega$ and $C_1>0$ be as in \hyperref[Linfinity bound near the boundary]{Lemma \ref{Linfinity bound near the boundary}}. By \eqref{f0}, \hyperref[Linfinity bound near the boundary]{Lemma \ref{Linfinity bound near the boundary}}, \hyperref[L1 uniform with phi1]{Lemma \ref{L1 uniform with phi1}}, and the positivity of $\varphi_1$, we get
    \begin{align*}
        \int_{B}f(x,u)\,\mathrm{d}x&=\int_{\omega} f(x,u)\,\mathrm{d}x+\int_{B \setminus \omega} f(x,u)\,\mathrm{d}x\\
        &\leq \sup_{(x,t) \in \overline{\omega}\times [0,C_1]}f(x,t)\int_{\omega}\,\mathrm{d}x+\frac{1}{\min_{B\setminus \omega}\varphi_1}\int_{B\setminus\omega} f(x,u)\varphi_1\,\mathrm{d}x\\
        &\leq \sup_{(x,t) \in \overline{\omega}\times [0,C_1]}f(x,t)|B|+\frac{C}{\min_{B}\varphi_1}
        \coloneq M(f,B).
    \end{align*}
\end{proof}

\section{A priori estimates for the scalar problem in the subcritical regime}\label{a priori subcrit}
Throughout this section, we assume that $f$ satisfies \eqref{f1'}. The goal is to provide a proof of \hyperref[a priori estimate2 final]{Theorem \ref{a priori estimate2 final}}. Our approach is based on the following inequalities.
\begin{lemma}{\cite[Lemma 2.4]{criticalandsubcritical}}\label{real number inequality1} For all $t \geq 0$, we have
\begin{equation*}
  st \leq
     \begin{cases} 
      e^{t^2}-1+s(\log s)^{\frac{1}{2}}, &\hspace{0.74cm}  s\geq e^{1/4}, \\
      e^{t^{2}}-1+\frac{s^{2}}{2}, &  0\leq s<e^{1/4}. \\ 
   \end{cases}
   \end{equation*}
\end{lemma}

\noindent
The next inequality generalizes the previous one to the case of three terms. Although it was introduced in \cite[Proposition 4.2]{joaorufdjairo2}, we provide a proof here for the convenience of the reader.

\begin{lemma}\label{real number inequality2}
There exists a constant $C>0$ such that, for any $a,b \geq1$ and $r,s,t>0$, we have
$$rst\leq C\left(e^{t^{a}}+e^{r^{b}}+s(\log^{+} s)^{\frac{1}{a}+\frac{1}{b}}\right).$$
\end{lemma}
\begin{proof}
We split the proof into two cases. First assume that \(0<s\le 1\). Then \(\log^+s=0\), and
\[
rst\le rt\le e^r+e^t.
\]
Since \(b\ge 1\), we have \(e^r\le e\,e^{r^b}\) for every \(r>0\): indeed, if
\(r\ge 1\), then \(r^b\ge r\), while if \(0<r<1\), then \(e^r\le e\le e\,e^{r^b}\).
Analogously, \(e^t\le e\,e^{t^a}\) for every \(t>0\). Hence
\[
rst\le e\left(e^{r^b}+e^{t^a}\right),
\]
which proves the estimate when \(0<s\le 1\).
We now assume that \(s>1\). We may suppose, without lack of generality, that \(t^a\le r^b\). Hence \(t\le r^{b/a}\), and so
\[
rst\le r^{1+b/a}s.
\]
Set
\[
p:=\frac{ab}{a+b}, \qquad x:=r^{1+b/a}.
\]
Then \(p\ge 1/2\), because \(a,b\ge 1\), and
\[
x^p=\left(r^{1+b/a}\right)^{\frac{ab}{a+b}}=r^b,
\qquad
\frac1p=\frac1a+\frac1b.
\]
Therefore,
\[
rst\le xs,
\]
and it is enough to prove that
\[
R_1:=\frac{xs}{e^{x^p}+s(\log s)^{1/p}}
\]
is bounded by a constant independent of \(r,s,t\).
Assume first that
\[
x\le 2^{1/p}(\log s)^{1/p}.
\]
Then
\[
R_1
\le
\frac{2^{1/p}s(\log s)^{1/p}}
     {e^{x^p}+s(\log s)^{1/p}}
\le 2^{1/p}\le 4,
\]
since \(p\ge 1/2\).
Assume now that
\[
x>2^{1/p}(\log s)^{1/p}.
\]
Then \(x^p>2\log s\), hence \(s<e^{x^p/2}\). Consequently,
\[
R_1
\le
\frac{xs}{e^{x^p}}
\le
xe^{-x^p/2}.
\]
If \(0<x<1\), then \(xe^{-x^p/2}\le 1\). If \(x\ge 1\), since \(p\ge 1/2\),
we have \(x^p\ge \sqrt{x}\), and therefore
\[
xe^{-x^p/2}\le xe^{-\sqrt{x}/2}.
\]
The function \(x\mapsto xe^{-\sqrt{x}/2}\) is bounded on \([1,+\infty)\).
Thus \(R_1\le C\) in this second case as well.
Combining the two cases, we obtain
\[
xs\le C\left(e^{x^p}+s(\log s)^{1/p}\right).
\]
Recalling that \(x=r^{1+b/a}\), \(x^p=r^b\), and
\(1/p=1/a+1/b\), we conclude that
\[
rst
\le C\left(e^{r^b}+s(\log s)^{\frac1a+\frac1b}\right)
\le
C\left(e^{t^a}+e^{r^b}
+s(\log s)^{\frac1a+\frac1b}\right)
\]
whenever \(s>1\) and \(t^a\le r^b\).
If \(r^b\le t^a\), the same argument with the roles of \((r,b)\) and
\((t,a)\) interchanged gives the same estimate. Therefore, for every
\(r,s,t>0\),
\[
rst
\le
C\left(
e^{t^a}+e^{r^b}
+s(\log^+s)^{\frac1a+\frac1b}
\right)
\]
and the proof is complete.
\end{proof}

\noindent
We also recall Adams' inequality (see \cite{Adams}).
\begin{theorem}\label{Adams inequality}If $\Omega \subseteq \mathds{R}^4$ is a smooth bounded domain, then
\begin{equation*}
    \sup_{\substack{u \in H_0^{2}(\Omega)\\\\\lVert \Delta u\lVert_{L^2(\Omega)\leq 1}}}\int_{\Omega}e^{32\pi^2u^2(x)}\,\mathrm{d}x< \infty.
\end{equation*}
\end{theorem}

With the help of the aforementioned inequalities, and adapting the method of proof introduced by de Figueiredo, do \'O, and Ruf \cite[Section 4]{joaorufdjairo2} for second-order elliptic systems in dimension two,  we obtain the following uniform  Sobolev estimate for the fourth-order Dirichlet problem \eqref{App}.
\begin{theorem}\label{Sobolev uniform estimate} Suppose that $\eqref{f1'}$, $\eqref{f2}$, and  $\eqref{f3}$ hold. Then there exists a constant $C>0$ such that $\lVert \Delta u\rVert_{L^2(B)}\leq C$, for  every very weak solution $u$ of \eqref{App}.
\end{theorem}
\begin{proof} Let $u$ be a very weak solution of \eqref{App}. By \hyperref[regularity1]{Theorem \ref{regularity1}}, $u$ is in particular a weak solution of \eqref{App}, thus
$$
    \int_{B} |\Delta u|^2\,\mathrm{d}x=\int_{B} u f(x,u)\,\mathrm{d}x
    \leq C+ \int_{B'}\frac{f(x,u)}{1+u^{\eta}}(1+u^{\eta})u\,\mathrm{d}x
    =C+I_1+I_2,
$$
where $B'\coloneq \{x \in B: u(x)\geq 1\}$,  
\begin{equation*}
    I_1\coloneq\int_{B'} \frac{f(x,u)}{1+u^{\eta}}u\,\mathrm{d}x,\quad  I_2\coloneq \int_{B'} \frac{f(x,u)}{1+u^{\eta}}u^{\eta+1}\,\mathrm{d}x,
\end{equation*}
and $\eta>0$ will be chosen during the proof. We denote $\lVert u\rVert\coloneq \lVert \Delta u\lVert_{L^2(B)}$ and estimate $I_1$ by applying \hyperref[real number inequality1]{Lemma \ref{real number inequality1}} in the following way
\begin{align*}
    I_1&= \lVert u\rVert \int_{B'} \frac{f(x,u)}{1+u^{\eta}}\left(\frac{u}{\lVert u \rVert}\right)\,\mathrm{d}x\\
    &\leq \lVert u \rVert \left \{2\int_{B'}e^{\left(\frac{u}{\lVert u\rVert}\right)^{2}}\,\mathrm{d}x+\int_{B''_{1}}\frac{f(x,u)}{1+u^{\eta}}\left(\log \frac{f(x,u)}{1+u^{\eta}}\right)^{1/2}\,\mathrm{d}x+\frac{1}{2}\int_{B''_2}\left(\frac{f(x,u)}{1+u^{\eta}}\right)^2\,\mathrm{d}x\right\},
\end{align*}
where
\begin{equation*}
   B''_1 \coloneq B' \cap \left \{x \in B: \frac{f(x,u)}{1+u^{\eta}}> e^{1/4}\right\}\quad \text{and}\quad  B''_{2} \coloneq B' \cap \left \{x \in B: \frac{f(x,u)}{1+u^{\eta}}\leq e^{1/4}\right\}.
\end{equation*}
Then, we apply Adams' inequality \hyperref[Adams inequality]{Theorem \ref{Adams inequality}}, to obtain
\begin{equation}\label{adams app}
    I_1 \leq \lVert u \rVert \left\{C+\int_{B''_{1}}\frac{f(x,u)}{1+u^{\eta}}\left(\log \frac{f(x,u)}{1+u^{\eta}}\right)^{1/2}\,\mathrm{d}x\right\}.
\end{equation}
First, we estimate the integral on the right hand side of \eqref{adams app} using $\eqref{f1'}$ and \hyperref[uniform L11]{Theorem \ref{uniform L11}} as follows:
\begin{align*}
    \int_{B''_1} \frac{f(x,u)}{1+u^{\eta}}\left(\log \frac{f(x,u)}{1+u^{\eta}}\right)^{1/2}\,\mathrm{d}x
    &\leq  \int_{B''_1}\frac{f(x,u)} {1+u^{\eta}}(\log f(x,u))^{1/2}\,\mathrm{d}x
    \leq \int_{B''_1}\frac{f(x,u)}{1+u^{\eta}}(\log (Ce^{u^{\alpha}}))^{1/2}\,\mathrm{d}x\\
    &\leq\int_{B''_1}\frac{f(x,u)}{1+u^{\eta}}(\log \max\{C,1\}+u^{\alpha})^{1/2}\,\mathrm{d}x\\
    &\leq(\log \max\{C,1\}+1)^{1/2} \int_{B_1''}\frac{f(x,u)}{1+u^{\eta}}u^{\alpha/2}\,\mathrm{d}x\\
    &\leq C_0\int_{B''_1}\frac{f(x,u)}{1+u^{\eta}}(u^{\alpha/2}+1)\,\mathrm{d}x
    \leq C_0\int_{B''_1} f(x,u)\,\mathrm{d}x
    \leq C_0M,
\end{align*}
provided
\begin{equation}\label{condition1}
\eta \geq \frac{\alpha}{2}.
\end{equation}
Summarizing, we just obtained
\begin{equation}\label{I1 estimate complete}
    I_1\leq C(\lVert u \rVert +1)\quad \text{provided}\quad \eta\geq \frac{\alpha }{2}.
\end{equation}
Now, we estimate $I_2$ using \hyperref[real number inequality2]{Lemma \ref{real number inequality2}}:
$$
    I_2=\lVert u \rVert^{\eta+1}\int_{B'}\frac{f(x,u)}{1+u^{\eta}}\left(\frac{u}{\lVert u\rVert}\right)^{\eta}\frac{u}{\lVert u\rVert}\,\mathrm{d}x
    \leq C \lVert u\rVert^{\eta+1}\left\{\int_{B'}e^{\left(\frac{u}{\lVert u\rVert}\right)^{a\eta}}\,\mathrm{d}x+\int_{B'}e^{\left(\frac{u}{\lVert u\rVert}\right)^{b}}\,\mathrm{d}x+J\right\},
$$
where 
\begin{equation*}
    J\coloneq \int_{B'} \frac{f(x,u)}{1+u^{\eta}}\left(\log^{+}\frac{f(x,u)}{1+u^{\eta}}\right)^{\frac{1}{a}+\frac{1}{b}}\,\mathrm{d}x .
\end{equation*}
Next, we can choose $a,b\geq 1$ such that
\begin{equation}\label{condition2}
    a\eta\leq 2\quad \text{and}\quad b\leq 2,
\end{equation}
and apply Adams' inequality \hyperref[Adams inequality]{Theorem \ref{Adams inequality}} in order to obtain
\begin{equation}\label{adams app2}
    I_{2}\leq C\lVert u\rVert^{\eta+1}(C+J).
\end{equation}
Finally, the integral $J$ in \eqref{adams app2} can be estimated as above:
\begin{align*}
J&\leq C\int_{B'} \frac{f(x,u)}{1+u^{\eta}}u^{\alpha(a^{-1}+b^{-1})}\,\mathrm{d}x
\leq C\int_{B'} \frac{f(x,u)}{1+u^{\eta}}\left(u^{\alpha(a^{-1}+b^{-1})}+1\right)\,\mathrm{d}x
\leq C\int_{B'} f(x,u)\,\mathrm{d}x
\leq CM,
\end{align*}
provided
\begin{equation}\label{condition3}
    \alpha(a^{-1}+b^{-1})\leq \eta.
\end{equation}
Thus, we just proved that
\begin{equation}\label{I2 estimate complete}
    I_2\leq C\lVert u \rVert^{\eta+1}\quad \text{provided}\quad \alpha(a^{-1}+b^{-1})\leq \eta,\hspace{0.1cm}a\eta\leq 2\hspace{0.1cm }\text{and }b\leq 2.
\end{equation}
Hence, from \eqref{I1 estimate complete} and \eqref{I2 estimate complete} we get
\begin{align*}
    \lVert u\rVert^{2}=C+I_1+I_2\leq C(\lVert u\rVert+\lVert u\rVert^{\eta+1}+1),
\end{align*}
which implies
\begin{equation}\label{final estimate}
    \lVert u \rVert \leq 2\max\{C,1\}(1+\lVert u \rVert^{\eta}).
\end{equation}
Now, under conditions \eqref{condition1}, \eqref{condition2} and \eqref{condition3} we have that
\begin{equation*}
    \eta\geq \alpha/2, \hspace{0.2cm} a^{-1}+b^{-1}\geq (\eta+1)/2
\hspace{0.2cm }\text{and}\hspace{0.2cm} a^{-1}+b^{-1}\leq \eta/\alpha,
\end{equation*}
which implies
\begin{equation*}
    \frac{\eta+1}{2}\leq \frac{\eta}{\alpha}, \quad \text{ or equivalently } \quad \eta \geq \frac{\alpha}{2-\alpha}.
\end{equation*}
We choose equality in the relation above and since $\alpha \in (0,1)$, we have $\eta \in (0,1)$, and the desired uniform estimate follows directly by \eqref{final estimate}.
\end{proof}

The above uniform Sobolev estimate enables us to show that the global uniform $L^1$ bound on the nonlinearity, obtained in \hyperref[uniform L11]{Theorem \ref{uniform L11}}, can be improved in the subcritical regime \eqref{f1'} to a uniform $L^p$ bound for every $p>1$.

\begin{theorem}\label{Lp uniform estimate} Suppose that $\eqref{f1'}$, $\eqref{f2}$, and  $\eqref{f3}$ hold. If $p>1$, then there exists a constant $M_p>0$ such that
$$\int_B (f(x,u))^p \, \mathrm{d}x \leq M_p,$$
for every very weak solution $u$ of \eqref{App}.
\end{theorem}
\begin{proof}
Let $u$ be a very weak solution of \eqref{App}. By \hyperref[Sobolev uniform estimate]{Theorem \ref{Sobolev uniform estimate}}, there exists $M>0$, independent of $u$, such that $\lVert \Delta u\rVert_{L^2(B)}\leq M$.

Let $\alpha \in (0,1)$ be as in $\eqref{f1'}$. Then, for every $\mu>0$, there exists a constant $C_{p,\mu}>0$ such that
$$
    e^{pt^{\alpha}}\leq C_{p,\mu}e^{\mu t^2}\quad  \forall t\geq 0.
$$
Combining this estimate with \eqref{f1'}, we get
\begin{equation*}
     \int_B (f(x,u))^p \, \mathrm{d}x \leq C^p \int_B e^{pu^{\alpha}} \, \mathrm{d}x \leq C^p \cdot C_{p,\mu} \int_B e^{\mu u^{2}} \, \mathrm{d}x\leq C^p \cdot C_{p,\mu} \int_B e^{\mu M^2\frac{u^2}{\lVert \Delta u\rVert_{L^2(B)}^2}} \, \mathrm{d}x.
\end{equation*}
Choosing $\mu>0$ small enough such that $M^2\mu \leq 32\pi^2$, the result follows directly from Adams' inequality \hyperref[Adams inequality]{Theorem \ref{Adams inequality}}.
\end{proof}

We conclude this section with the proof of \hyperref[a priori estimate2 final]{Theorem \ref{a priori estimate2 final}}, which is a direct consequence of \hyperref[Lp uniform estimate]{Theorem \ref{Lp uniform estimate}}. Indeed, in view of \hyperref[Lp uniform estimate]{Theorem \ref{Lp uniform estimate}}, we can apply
standard $L^p$ estimates (see  \cite[Corollary 2.21]{Gazzola}):
\begin{equation*}
    \lVert u\rVert_{L^{\infty}(B)}\leq C\lVert f(x,u)\rVert_{L^{p}(B)}\leq C,
\end{equation*}
completing the proof of \hyperref[a priori estimate2 final]{Theorem \ref{a priori estimate2 final}}.

\section{A priori estimates for the scalar problem in the critical regime}\label{subsec a priori crit}

Throughout this section, we assume that $f$ satisfies \eqref{f1}, \eqref{f2}, and \eqref{f3}. Our goal is to prove \hyperref[a priori estimate1 final]{Theorem \ref{a priori estimate1 final}}, namely to obtain a priori estimates for very weak solutions of \eqref{App} in the critical regime \eqref{f1}. Unlike in the subcritical regime, here we cannot expect to improve the global uniform $L^1$ bound on the nonlinearity to a uniform $L^p$ bound, since \eqref{f1} imposes a critical exponential growth on $f(x,t)$. Therefore, we follow a different strategy inspired by the blow-up argument performed in \cite[Section 5]{joaorufdjairo2} for second-order elliptic systems in dimension two.

Assume that $(u_n)$ is a sequence of solutions of \eqref{App}. \hyperref[uniform L11]{Theorem \ref{uniform L11}} ensures that the sequence $(f(x,u_n))$ is bounded in $L^1(B).$ So it follows, passing to a subsequence if necessary (see \cite[page 115 and Problem 24]{BrezisFA}) that there is a bounded Radon measure $\mu$ such that
\begin{equation}\label{measure convergence}
    f(x,u_n) \rightharpoonup \mu.
\end{equation}
More precisely
\begin{equation}\label{measure convergence def}
    \int_{B} f(x,u_n)\varphi\,\mathrm{d}x\to \int_{B} \varphi\,\mathrm{d}\mu\quad \forall \varphi \in C_c(B).
\end{equation}
Since $f$ is a non-negative function it follows that $\mu$ is a non-negative measure. We also observe that, as a consequence of \hyperref[solution]{Lemma \ref{solution}}, \hyperref[estimates]{Lemma \ref{estimates}}, and  \hyperref[uniform L11]{Theorem \ref{uniform L11}}, the sequence of solutions $(u_n)$ of \eqref{App} is uniformly bounded in $L^1(B)$. Indeed,
\begin{align*}
    \lVert u_n\rVert_{L^{1}(B)}=\int_{B} \int_{B} G(x,y)f(y,u_n(y))\,\mathrm{d}y\,\mathrm{d}x&\leq \frac{1}{8\pi^2}\int_{B}\int_{B} \log \left(\frac{2}{|x-y|}\right)f(y,u_n(y))\,\mathrm{d}y\,\mathrm{d}x\\
    &= \frac{1}{8\pi^2}\int_{B}\int_{B} \log \left(\frac{2}{|x-y|}\right)f(y,u_n(y))\,\mathrm{d}x\,\mathrm{d}y\\
    &\leq\frac{1}{8\pi^2} \int_{B}\log \left(\frac{2}{|x|}\right)\,\mathrm{d}x\,\lVert f(x,u_n)\rVert_{L^{1}(B)},
\end{align*}
that is,
\begin{equation}\label{L1 bound of solutions}
    \lVert u_n\rVert_{L^1(B)}\leq C,\quad \forall n \in \mathds{N}.
\end{equation}
\begin{definition}\label{regular point def} We say that $x_0 \in B$ is a regular point of the measure $\mu$ if there is a function $\psi \in C_c(B)$, $0 \leq \psi \leq1$ in $B$, with $\psi \equiv 1$ in some neighborhood $x_0$, such that
\begin{equation*}
    \int \psi\,\mathrm{d}\mu<32\pi^2.
\end{equation*}
We denote $\Sigma_{\mu}$ the set of non-regular points in $B$ for the measure $\mu.$
\end{definition}
\begin{remark}\label{non-regular points finite} The set $\Sigma_{\mu}$ is finite. Indeed, if $x_0 \in \Sigma_{\mu}$, we have that
\begin{equation*}
    \int \psi\,\mathrm{d}\mu\geq 32\pi^2,
\end{equation*}
for each function $\psi \in C_c(B)$, $0 \leq \psi \leq1$ in $B$, with $\psi \equiv 1$ in some neighborhood of $x_0.$ Thus, $\mu(\{x_0\}) \geq 32\pi^2$. Finally, since $\mu (B)=\int\,\mathrm{d}\mu<+\infty$, it follows that $\Sigma_{\mu}$ is a finite set.
\end{remark}
\begin{definition}\label{blow up set} We define the blow-up set $S$ for the sequence of solutions $(u_n)$ of \eqref{App} as
\begin{equation*}
    S\coloneq \left\{x \in B:\text{there exists }(x_n)\subseteq B\hspace{0.1cm}\text{such that }x_n \to x\hspace{0.1cm}\text{and }u_n(x_n)\to+\infty \right\}.
\end{equation*}
\end{definition}
The objective of the subsequent lemmas is to establish that, under the assumptions of \hyperref[a priori estimate1 final]{Theorem \ref{a priori estimate1 final}}, we have $S=\emptyset$ (see \hyperref[a priori crit]{Theorem \ref{a priori crit}} below). Once this is achieved, the proof of \hyperref[a priori estimate1 final]{Theorem \ref{a priori estimate1 final}} follows immediately. Indeed, if the desired uniform \(L^\infty\) estimate failed, then some sequence of solutions would blow up, yielding a nonempty blow-up set $S$, which contradicts \hyperref[a priori crit]{Theorem \ref{a priori crit}} below.

\begin{lemma}\label{a priori local} Suppose that \eqref{f0} with $p=1$, \eqref{f2}, and  \eqref{f3} hold. If $x_0$ is a regular point for the measure $\mu$, then there exist constants $\rho>0$ and $M>0$, independent of $n$, such that
 \begin{equation*}
     \lVert u_n\rVert_{L^{\infty}(B(x_0,\rho))}\leq M.
 \end{equation*}   
\end{lemma}
\begin{proof} Using the fact that $x_0$ is a regular point of the measure $\mu,$ we have a function $\psi \in C_c(B)$, $0\leq \psi \leq 1$ in $B$, with $\psi \equiv 1$ in some neighborhood $V_{x_0}$ of $x_0$, such that $\int \psi\,\mathrm{d}\mu<32\pi^2.$ Thus $\int_{V_{x_0}}\,\mathrm{d}\mu<32\pi^2$, which implies that there exist $R>0$, $\delta\in(0,32\pi^2)$, and $n_0 \in \mathds{N}$ such that $B(x_0,R) \subset \subset B$ and for all $n \geq n_0$
\begin{equation}\label{local a priori 1}
    \int_{B(x_0,R)} f(x,u_n)\,\mathrm{d}x\leq 32\pi^2-\delta.
\end{equation}
Now, for each $n\geq n_0$, write $u_n=u_{1,n}+u_{2,n}$, where
 \begin{equation}
\left\{\begin{array}{rcl}
\Delta^2 u_{1,n}  =  f(x,u_n) \ \ &\text{in} \ \ B(x_0,R),\\
u_{1,n}  =  \partial_{\nu}u_{1,n} = 0\ \ &\text{on} \ \ \partial B(x_0,R).\\
\end{array}\right.
\end{equation}
and $\Delta^2 u_{2,n}=0$ in $B(x_0,R).$ Next, we apply the Brezis-Merle inequality \eqref{BM-ball-R} and obtain
\begin{equation}\label{local a priori}
    C\geq \int_{B(x_0,R)} e^{(32\pi^2-\delta/2)\frac{|u_{1,n}(x)|}{\lVert f\rVert_{L^1(B(x_0,R))}}}\,\mathrm{d}x\geq \int_{B(x_0,R)} e^{(32\pi^2-\delta/2)\frac{|u_{1,n}(x)|}{32\pi^2-\delta}}\,\mathrm{d}x=\int_{B(x_0,R)}e^{q|u_{1,n}(x)|}\,\mathrm{d}x,
\end{equation}
where $q\coloneq (32\pi^2-\delta/2)/(32\pi^2-\delta)>1$. Using the fact that $t \leq e^{t}$ for all $t \in \mathds{R},$ we get
\begin{equation}\label{local a priori 2}
    \lVert u_{1,n}\rVert_{L^{q}(B(x_0,R))}\leq C \quad \forall n\geq n_0.
\end{equation}
Since the function $u_{2,n}$ is biharmonic in $B(x_0,R)$, we can apply classical interior estimates, as in \cite[Section 2.5.2]{Gazzola}, and obtain
\begin{equation*}
    \lVert u_{2,n}\rVert_{L^{\infty}(B(x_0,R/2))}\leq C\lVert u_{2,n}\rVert_{L^1({B(x_0,R)})}.
\end{equation*}
On the other hand, using \eqref{L1 bound of solutions} and \eqref{local a priori 2}, we deduce
\begin{equation*}
    \lVert u_{2,n}\rVert_{L^1(B(x_0,R))}\leq \lVert u_n\rVert_{L^1(B(x_0,R))}+\rVert u_{1,n}\rVert_{L^1(B(x_0,R))}\leq C,
\end{equation*}
and so
\begin{equation}\label{local a priori 3}
    \lVert u_{2,n}\rVert_{L^{\infty}(B(x_0,R/2))}\leq C.
\end{equation}
Let us now use assumption $\eqref{f0}$ with $p=1$. Then $f(x,u_n)\leq ce^{u_{1,n}+u_{2,n}}=ce^{u_{1,n}}e^{u_{2,n}}$. So it follows from \eqref{local a priori} and \eqref{local a priori 3} that
\begin{equation}\label{local a priori 4}
    \lVert f(x,u_n)\rVert_{L^q(B(x_0,R/2))}\leq \tilde{C}\left\lVert e^{|u_{1,n}|}\right\rVert_{L^q(B(x_0,R/2))}\leq C, \quad \text{for some }q>1.
\end{equation}
Thus, again by classical interior estimates (see e.g. \cite[Section 2.5.2]{Gazzola}), \eqref{local a priori 2} and \eqref{local a priori 4}, we get
\begin{equation*}
    \lVert u_{1,n}\rVert_{L^{\infty}(B(x_0,R/4))}\leq C\left(\lVert f(x,u_n)\rVert_{L^{q}(B(x_0,R/2))}+\lVert u_{1,n}\rVert_{L^1(B(x_0,R/2))}\right)\leq C,
\end{equation*}
and since $u_{1,n}$ and $u_{2,n}$ are both uniformly bounded in $L^{\infty}(B(x_0,R/4))$, this implies our claim with $\rho=R/4$.
\end{proof}
\begin{lemma}\label{relation of regular and blowup sets} Suppose that  \eqref{f0} with $p=1$, \eqref{f2}, and  \eqref{f3} hold. Then $\Sigma_{\mu}=S$.
\end{lemma}
\begin{proof}The inclusion $S \subseteq \Sigma_{\mu}$ follows directly from \hyperref[a priori local]{Lemma \ref{a priori local}}. It remains to prove that $\Sigma_{\mu} \subseteq S$. Let $x_0 \in \Sigma_{\mu}$. We claim that for each $R>0$ we have
\begin{align}\label{claim}
    \lim_{n \to +\infty} \lVert u_n\rVert_{L^{\infty}(B(x_0,R))}=+\infty.
\end{align}
Suppose by contradiction that there exist $R_0>0$ and a subsequence, still denoted by $(u_n)$, such that
\begin{align*}
    \lVert u_n \rVert_{L^{\infty}(B(x_0,R_0))}\leq C.
\end{align*}
Hence, using hypothesis $\eqref{f0}$ with $p=1$, it follows that
\begin{align*}
    \lVert f(x,u_n)\rVert_{L^{\infty}(B(x_0,R_0))} \leq c \lVert e^{u_n} \rVert_{L^{\infty}(B(x_0,R_0))} \leq C,
\end{align*}
which implies that for $R<R_0$ we have
\begin{align*}
    \int_{B(x_0,R)} f(x,u_n)\,\mathrm{d}x\leq CR^4.
\end{align*}
Thus, there exists $R_1>0$, such that
\begin{align*}
    \int_{B(x_0,R_1)} f(x,u_n)\,\mathrm{d}x<32\pi^2.
\end{align*}
This implies that $x_0$ is a regular point of $\mu$, which is a contradiction.

Now we observe that, using \hyperref[non-regular points finite]{Remark \ref{non-regular points finite}}, there exists $R>0$ such that $x_0$ is the only non-regular point in $\overline{B(x_0,R)}$. Next, we use \eqref{claim} to prove that $x_0 \in S.$ Indeed, from \eqref{claim} there exists $(x_n) \subseteq B(x_0,R)$ such that $x_n \to \tilde{x}$ and $u_n(x_n)\to +\infty$. So, we need to prove $\tilde{x}=x_0$. Indeed, $\tilde{x} \in S \cap \overline{B(x_0,R)}\subseteq \Sigma_{\mu} \cap \overline{B(x_0,R)}$, implying  $x_0=\tilde{x}$, since 
$
    \Sigma_{\mu} \cap \overline{B(x_0,R)}=\{x_0\}
$.
\end{proof}
In order to show that $S=\emptyset$, we will need the following estimate from below of the Green function.
\begin{lemma}\label{Green Below} For each $x_0 \in B$, there exists $\rho>0$ such that every $x \in B(x_0,\rho)\setminus \{x_0\}$ satisfies
\begin{equation*}
    G(x,x_0) \geq \frac{1}{8\pi^2}\log\frac{1}{|x-x_0|}-C_{x_0}, \quad \text{ where } C_{x_0}\coloneq \frac{1}{8\pi^2}\left(\frac{1}{2}-\log \frac{1-|x_0|^2}{2}\right).
\end{equation*}
\end{lemma}
\begin{proof}
Recalling \eqref{Green2}, \eqref{Green3}, and
\begin{equation*}
    -\frac{1}{2}\left(1-\frac{|x-y|^2}{[XY]^2}\right)=-\frac{1}{2}+\frac{1}{2}\frac{|x-y|^2}{[XY]^2}\geq-\frac{1}{2},
\end{equation*}
we have
\begin{equation}\label{Green Below1}
    G(x,y)=\frac{1}{8\pi^2}\left[\log\frac{[XY]}{|y-x|}-\frac{1}{2}\left(1-\frac{|y-x|^2}{[XY]^2}\right)\right] \geq \frac{1}{8\pi^2}\left(\log \frac{[XY]}{|x-y|}-\frac{1}{2}\right).
\end{equation}
Now, fix $x_0 \in B$ and observe that $[X_0X_0]=1-|x_0|^2>0$. By the continuity of the map $x\mapsto [XX_0]$, there exists $\rho>0$ such that for every $0<|x-x_0|<\rho$ we have $[XX_0]\geq (1-|x_0|^2)/2$. This inequality and \eqref{Green Below1} with $y=x_0$ yield
\begin{equation*}
    G(x,x_0) \geq \frac{1}{8\pi^2} \left( \log \frac{1-|x_0|^2}{2|x-x_0|} - \frac 1 2 \right) = \frac{1}{8\pi^2}\log\frac{1}{|x-x_0|}-C_{x_0}\quad \forall x \in B \cap B(x_0,\rho).
\end{equation*}
\end{proof}
\begin{theorem}\label{a priori crit} Suppose that \eqref{f1}, \eqref{f2}, and \eqref{f3} hold. Then $S=\emptyset$.
\end{theorem}
\begin{proof} Suppose, by contradiction, that there exists $x_0 \in S$. Since $x_0$ is isolated, we can take $R>0$ such that $\overline{B(x_0,R)} \cap (S \setminus \{x_0\})=\emptyset$. Suppose that $R>0$ is small enough so that $B(x_0,2R) \subset \subset B$ and consider $\eta_R \in C_c(B)$ satisfying
\begin{equation*}
    0\leq \eta_R \leq 1,\quad \eta_{R} \equiv 1\hspace{0.1cm}\text{in }\overline{B(x_0,R)}, \quad \mathrm{supp}\hspace{0.1cm}\eta_R\subseteq B(x_0,2R).
\end{equation*}
Now, for each $n \in \mathds{N}$, we consider the following problem
\begin{equation*}
\left\{
\begin{array}{rcll}
\Delta^{2} z_n &=& f(x,u_n)\eta_R\, & \text{in } B,\\
z_n=\partial_\nu z_n &=& 0 & \text{on } \partial B.
\end{array}
\right.
\end{equation*}
By the Green representation formula  \eqref{Green representation}, we have
\begin{equation*}
    z_n(x)=\int_{B} G(x,y)\eta_R(y)f(y,u_n(y))\,\mathrm{d}y,
\end{equation*}
and since $G,f \geq 0$ we see that $z_n \geq 0$ for every $n.$ Similarly, $u_n$ is a solution of \eqref{App}. thus
\begin{equation*}
    u_n(x)=\int_{B} G(x,y) f(y,u_n(y))\,\mathrm{d}y
\end{equation*}
and since $0 \leq \eta_R \leq 1$ we get
\begin{equation}\label{comp zn and un}
    0 \leq z_n(x)\leq u_n(x)\quad \text{in } B. 
\end{equation}
Next, fix $x \in B$, $k \in \mathds{N}$, and define
\begin{equation*}
    \phi_{x,k}(y)\coloneq \min\{G(x,y),k\}.
\end{equation*}
We observe that $\phi_{x,k} \in C(\overline{B})$, $0 \leq  \phi_{x,k}\leq G(x,\cdot)$, and $\eta_R \phi_{x,k}\in C_c(B)$. Then, we can apply \eqref{measure convergence def} with $\varphi = \eta_R \phi_{x,k}$ to obtain
\begin{equation}\label{measure limit}
    \int_{B} \eta_R(y)\phi_{x,k}(y)f(y,u_n(y))\,\mathrm{d}y\to \int_{B} \eta_R(y)\phi_{x,k}(y)\,\mathrm{d}\mu.
\end{equation}
By $0 \leq \phi_{x,k}\leq G(x,\cdot)$ we can also derive
\begin{equation}\label{liminf victim}
    z_n(x)=\int_{B} G(x,y)\eta_R(y)f(y,u_n(y))\,\mathrm{d}y\geq \int_{B} \phi_{x,k}(y)\eta_R(y)f(y,u_n(y))\,\mathrm{d}y.
\end{equation}
Thus, passing to the $\liminf$ into \eqref{liminf victim} and using \eqref{measure limit}, we get
\begin{equation}\label{almost liminf}
    \liminf_{n \to +\infty} z_n(x)\geq \int_{B} \phi_{x,k}(y)\eta_R(y)\,\mathrm{d}\mu \quad \forall k \in \mathds{N}.
\end{equation}
The sequence of functions $(\phi_{x,k})_k$ is monotone non-decreasing and it converges to $G(x,\cdot)$; moreover, $\eta_R \geq 0$. Hence, we can let $k \to +\infty$ in \eqref{almost liminf} and apply the monotone convergence theorem to get
\begin{equation}\label{liminf}
    \liminf_{n \to +\infty}z_n(x) \geq z(x)\hspace{0.2cm}\forall x \in B,\quad \text{where}\quad z(x)\coloneq \int_{B} \eta_R(y) G(x,y)\,\mathrm{d}\mu.
\end{equation}
By \eqref{liminf}, the continuity and monotonicity of $t \mapsto e^{t}$, we see that
\begin{equation*}
    e^{z(x)}\leq \liminf_{n \to +\infty}e^{z_n(x)}.
\end{equation*}
Thus, by Fatou's Lemma
\begin{equation}\label{Fatou}
    \int_{B(x_0,R)} e^{z(x)}\,\mathrm{d}x\leq \liminf_{n \to +\infty}\int_{B(x_0,R)} e^{z_n(x)}\,\mathrm{d}x.
\end{equation}

Our next objective is to show that 
\begin{equation*}
    \int_{B(x_0,R)} e^{z(x)}\,\mathrm{d}x=+\infty.
\end{equation*}
Indeed, since $\eta_R \equiv 1$ in $B(x_0,R)$, we get
\begin{equation}\label{z estimate 1}
    z(x)=\int_{B} \eta_{R}(y)G(x,y)\,\mathrm{d}\mu \geq \int_{B(x_0,R)} G(x,y)\,\mathrm{d}\mu.
\end{equation}
Now, fix $x \neq x_0$ in  $B(x_0,R)$ and consider $r>0$ such that $\overline{B(x_0,r)}\subseteq B(x_0,R)$ and $x \not \in B(x_0,R)$. In this way, the map $y \mapsto G(x,y)$ is continuous in $\overline{B(x_0,r)}$ and we can safely define
\begin{equation*}
    m_r(x) \coloneq \min_{y \in \overline{B(x_0,r)}} G(x,y).
\end{equation*}
Thus, recalling that $G \geq 0$ and $\mu$ is a non-negative measure, we arrive at
\begin{align*}
    \int_{B(x_0,r)} G(x,y)\,\mathrm{d}\mu\geq m_r(x)\mu(B(x_0,r)),
\end{align*}
and from \eqref{z estimate 1} we get
\begin{equation*}
    z(x) \geq m_r(x)\mu(B(x_0,r)).
\end{equation*}
Letting $r \to 0^{+}$, we conclude that
\begin{equation*}
    z(x) \geq G(x,x_0) \mu(\{x_0\}).
\end{equation*}
Moreover, since $x_0 \in S$, we have $\mu(\{x_0\}) \geq 32\pi^2$. By \hyperref[Green Below]{Lemma \ref{Green Below}}, we see that there exists $\rho>0$ such that $\overline{B(x_0,\rho)}\subseteq B(x_0,R)$ and
\begin{equation*}
    z(x) \geq 32\pi^2\left(\frac{1}{8\pi^2}\log\frac{1}{|x-x_0|}-C_{x_0}\right)=4\log\frac{1}{|x-x_0|}-\tilde{C}_{x_0}\quad \forall x\in B(x_0,\rho)\setminus\{x_0\},
\end{equation*}
and hence
\begin{equation}\label{below z near x0}
    e^{z(x)}\geq e^{-\tilde{C}_{x_0}}|x-x_0|^{-4}\quad \forall x \in B(x_0,\rho)\setminus\{x_0\}.
\end{equation}
Thus
$$
    \int_{B(x_0,R)} e^{z(x)}\,\mathrm{d}x\geq e^{-\tilde{C}_{x_0}}\int_{B(x_0,\rho)} |x-x_0|^{-4}
    =e^{-\tilde{C}_{x_0}}2\pi^3\int_{0}^{\rho}\frac{1}{r}\,\mathrm{d}r
    =+\infty,
$$
which proves our claim and, from \eqref{Fatou}
\begin{equation*}
    \liminf_{n\to +\infty}\int_{B(x_0,R)} e^{z_n(x)}\,\mathrm{d}x =+\infty.
\end{equation*}

Finally, using \eqref{comp zn and un} and hypothesis \eqref{f1}, we see that
\begin{equation*}
    f(x,u_n(x))\geq De^{u_n(x)}\geq De^{z_n(x)}\quad \text{in}\quad \{x \in B(x_0,R): u_n(x)\geq T\}.
\end{equation*}
Then, we have
\begin{align*}
    e^{z_n(x)}&=e^{z_n(x)}\chi_{\{u_n<T\}}+e^{z_n(x)}\chi_{\{u_n\geq T\}}\\
    &\leq e^{T}+\frac{1}{D}f(x,u_n(x))\quad \forall\, x \in B(x_0,R).
\end{align*}
Thus, integrating over $B(x_0,R)$ and taking the $\liminf$ in the previous inequality, we obtain
\begin{equation*}
    e^{T}|B(x_0,R)|+\liminf_{n \to +\infty}\frac{1}{D} \int_{B(x_0,R)} f(x,u_n)\,\mathrm{d}x\geq \liminf_{n \to +\infty}\int_{B(x_0,R)}e^{z_n(x)}\,\mathrm{d}x=+\infty,
\end{equation*}
which is a contradiction with the uniform $L^1$ bound established in \hyperref[uniform L11]{Theorem \ref{uniform L11}}.
\end{proof}

\section{Existence result for the scalar problem}\label{existence}

This section is devoted to the proof of \hyperref[existence of sol]{Theorem \ref{existence of sol}}. The proof relies on topological methods and heavily exploits the structural condition \eqref{f5}, as well as the uniform a priori bounds obtained in \hyperref[a priori estimate2 final]{Theorem \ref{a priori estimate2 final}} for the subcritical regime, and in \hyperref[a priori estimate1 final]{Theorem \ref{a priori estimate1 final}} for the critical regime. For this reason, there is no need to distinguish between the two regimes in what follows.

Adopting a well-established approach in the literature (see for instance \cite{dFLN1982}, \cite{joaorufdjairo2}, and \cite{Mancini}), we apply the following theorem of Krasnosel'ski\u{\i}, presented here in the form given by \cite[Theorem 3.1]{Djairolecturenotes}.
\begin{theoremletter}\label{topological method}
Let $C \subseteq E$ denote a cone in a Banach space $E$, and let $T: C \to C$ be a compact mapping with $T(0)=0$.
Assume furthermore that there exist real numbers $0<r<R$ such that
\begin{enumerate}
\item $x \neq t\,T x$ for $0\le t\le 1$, $\|x\|=r$, $x\in C$,
\item there exists $H: C\times[0,\infty)\to C$ such that
  \begin{enumerate}
  \item[(a)] $H_t:C \to C$ defined as $H_t(x) \coloneq H(x,t)$ is compact for every fixed $t \geq0$ and $H(x,0)=Tx$, \quad $x\in C$,
  \item[(b)] $H(x,t)=x$ has no solution for $x\in \overline{B}_R$, $t\geq t_0$,
  \item[(c)] $H(x,t)\neq x$, \quad $\|x\|=R$, $t\geq 0$,
  \end{enumerate}
\end{enumerate}
Then \emph{1.} implies $i_C(T,B_r)=1$, and \emph{2.} implies $i_C(T,B_R)=0$.
\end{theoremletter}

We observe that $i_C(T,\mathcal{O})$ is the fixed point index of $T$ in $C$ with respect to the (relatively) open set $\mathcal{O}$ (see \cite{Djairolecturenotes}). Note that it then follows by the excision property that $T$ has a fixed point in $B_R\setminus \overline{B}_r$.

For each $f \in C(\overline{B})$, from \hyperref[solution]{Lemma \ref{solution}}, there exists a unique very weak solution $u_f$ of \eqref{P}, given by
\begin{equation*}
    u_f(x)=\int_{B} G(x,y)f(y)\,\mathrm{d}y.
\end{equation*}
Moreover, as a consequence of \hyperref[regularity1]{Theorem \ref{regularity1}}, we see that $u_f \in C(\overline{B})$ and, from \hyperref[estimates]{Lemma \ref{estimates}} and \eqref{eq:Hardy-Littlewood}, we get
\begin{align*}
    |u_f(x)|&\leq \int_{B} G(x,y)|f(y)|\,\mathrm{d}y
    \leq \frac{\lVert f\rVert_{C(\overline{B})}}{8\pi^2}\int_{B}\log \frac{2}{|x-y|}\,\mathrm{d}y
    \leq \frac{\lVert f\rVert_{C(\overline{B})}}{8\pi^2}\int_{B} \frac{2}{|x-y|}\,\mathrm{d}y\\
    &\leq \frac{\lVert f\rVert_{C(\overline{B})}}{4\pi^2}\int_{B} \frac{1}{|y|}\,\mathrm{d}y
    =\frac{1}{6}\lVert f \rVert_{C(\overline{B})},
\end{align*}
that is, 
\begin{equation*}
    \lVert u_f \rVert_{C(\overline{B})}\leq \frac{1}{6}\lVert f\rVert_{C(\overline{B})}.
\end{equation*}
From the above considerations, we can consider a map $(\Delta^2)^{-1}:C(\overline{B})\to C(\overline{B})$, where for each $f \in C(\overline{B}$), $(\Delta^2 )^{-1}(f)$ is defined as the unique solution of \eqref{P}. With this definition, we clearly have that
\begin{equation}\label{InverseBiharmonic}
    \lVert (\Delta^2)^{-1}(f)\rVert_{C(\overline{B})}\leq \frac{1}{6}\rVert f\rVert_{C(\overline{B})}\quad \forall f \in C(\overline{B}).
\end{equation}

For the existence result, we require the additional hypothesis \eqref{f5}.
The proof of the existence result is inspired by de Figueiredo, do \'O, and Ruf \cite{joaorufdjairo2} and Mancini and Romani \cite{Mancini}.

\begin{proof}[Proof of Theorem \ref{existence of sol}] Consider the closed cone
\begin{equation*}
    C \coloneq \left\{u \in C(\overline{B}): u \geq0\hspace{0.1cm}\mathrm{in}\hspace{0.1cm}\overline{B}\right\}
\end{equation*}
and define the map $T:C \to C$ given by
\begin{equation*}
    T(u)(x)\coloneq (\Delta^2)^{-1}f(x,u)=\int_{B} G(x,y)f(y,u(y))\,\mathrm{d}y.
\end{equation*}
We observe that $T(C)\subseteq C$ because $f \geq 0$ and $G\geq 0$. Moreover, by \eqref{f5} we have $f(x,0)=0$ for every $x \in \overline{B}$, demonstrating that $T(0)=0$. By standard arguments involving the Green representation formula, regularity theory, and Ascoli-Arzel\`a theorem, $T$ is a compact operator. Next, we prove that the assumption \emph{1.}  of \hyperref[topological method]{Theorem \ref{topological method}} is satisfied.

\textbf{Proof of \emph{1.}}
Choose $ r \in (0,\varepsilon_0)$, where $\varepsilon_0$ is defined in $\eqref{f5},$ and suppose by contradiction that there exist $u \in C$, $\lVert u \rVert_{C(\overline{B})}=r$ and $ t \in [0,1]$ such that $u=tTu$. Then,
\begin{equation*}
    u=t(\Delta^2)^{-1}f(x,u),
\end{equation*}
and, taking the supremum over $\overline{B}$ and using $\eqref{f5}$, we get
\begin{equation*}
    \lVert u \rVert_{C(\overline{B})}\leq\frac{1}{6}\lVert f(x,u)\rVert_{C(\overline{B})}\leq \frac{\eta}{6}\lVert u\rVert_{C(\overline{B})}.
\end{equation*}
By assumption \eqref{f5}, $\eta/6<1$ and we reach a contradiction. Therefore, we conclude that $i_C(T, B_r)=1.$

\medskip

It remains to verify that the assumptions \emph{2.(a)-2.(c)} of \hyperref[topological method]{Theorem \ref{topological method}} are also satisfied. To this end, we introduce the homotopy $H: C \times [0,+\infty) \to C$ defined by 
$$H(u,t)\coloneq (\Delta^2)^{-1}f(x,u+t).$$ 

\textbf{Proof of \emph{2.(a).}} Since $G \geq 0$, we see that $H(C \times [0,+\infty))\subseteq C$. Moreover, we have $H(u,0)=(\Delta^2)^{-1}f(x,u)=Tu$ for every $u \in C$. Finally, by the same classical techniques mentioned earlier to justify the compactness of the operator $T$, $H_t$ is compact for each fixed $t \geq 0$.

\textbf{Proof of \emph{2.(b).}} Let $a,\varepsilon_1>0$ such that $\lambda_1+\varepsilon_1<a$. By $\eqref{f2}$ and \eqref{f2 implication}, there exist $C_0>0$ such that
\begin{equation}\label{final relation 1}
    f(x,s)\geq (\lambda_1+\varepsilon_1)s-C_0\quad \forall s\geq 0.
\end{equation}
We will prove that for each $u \in C$, there exists $t_0>0$ such that $H(u,t)=u$ has no solution for every $ t \geq t_0.$ Indeed, since $H(u,t)=u$ and arguing as in the proof of \hyperref[regularity1]{Theorem \ref{regularity1}}, we have that in particular $u$ is a weak solution of
\begin{equation*}
    \Delta^2 u =f(x,u+t)\quad \text{in}\quad B\quad \text{and}\quad u=\partial_{\nu}u=0\quad \text{on}\quad \partial B.
\end{equation*}
In this way,
\begin{equation}\label{final weak formulation}
    \int_{B}\Delta u\Delta \varphi\,\mathrm{d}x=\int_{B}f(x,u+t)\varphi\,\mathrm{d}x\quad \forall \varphi \in H_0^{2}(B).
\end{equation}
Now, consider $\varphi_1>0$ the first eigenfunction associated with the first eigenvalue $\lambda_1>0$ of $(\Delta^2, H_0^{2}(B)),$ that is,
\begin{equation}\label{phi1 relation}
    \int_{B}\Delta \varphi_1\Delta \psi\,\mathrm{d}x=\lambda_1\int_{B}\varphi_1\psi\,\mathrm{d}x\quad \forall \psi \in H_0^{2}(B).
\end{equation}
If we choose $\varphi =\varphi_1$ in \eqref{final weak formulation}, $\psi=u$ in \eqref{phi1 relation}, and $s=u+t$ in \eqref{final relation 1}, we obtain
\begin{equation*}
    \lambda_1\int_{B} u\varphi_1\,\mathrm{d}x=\int_{B} f(x,u+t)\varphi_1\,\mathrm{d}x \geq (\lambda_1+\varepsilon_1)\int_{B}(u+t)\varphi_1\,\mathrm{d}x-C_0\int_{B}\varphi_1\,\mathrm{d}x,
\end{equation*}
and, rearranging the terms,
\begin{equation*}
    \lambda_1\int_{B} u \varphi_1\,\mathrm{d}x \geq (\lambda_1+\varepsilon_1)\int_{B}u\varphi_1\,\mathrm{d}x+(\lambda_1+\varepsilon_1)t\int_{B}\varphi_1\,\mathrm{d}x-C_0\int_{B}\varphi_1\,\mathrm{d}x.
\end{equation*}
Thus
\begin{equation*}
    0 \geq -\varepsilon_1\int_{B}u\varphi_1\,\mathrm{d}x \geq (\lambda_1+\varepsilon_1)t\int_{B}\varphi_1\,\mathrm{d}x-C_0\int_{B}\varphi_1\,\mathrm{d}x,
\end{equation*}
which implies that
\begin{equation*}
    t \leq \frac{C_0}{\lambda_1+\varepsilon_1}.
\end{equation*}
The proof of \emph{2.(b)} is complete by choosing $t_0\coloneq 1+C_0/(\lambda_1+\varepsilon_1).$

\textbf{Proof of \emph{2.(c).}} Fix $t_0>0$ as in \emph{2.(b)}, so that $H(u,t)=u$ has no solution for all $t\geq t_0$.  
We prove that there exists $R>0$ such that
\[
H(u,t)\neq u\quad\forall\, u\in C,\ \|u\|_\infty=R,\ \text{and }\forall\,t\ge0.
\]

Since by \emph{2.(b)}, independently of $R>0$, there are no fixed points for $t\ge t_0$, it suffices to find $R>0$ such that there is no fixed points in the sphere $\|u\|_\infty=R$ for $t\in[0,t_0]$. To achieve our goal, for each $t\in[0,t_0],$ define
\[
f_t(x,s):=f(x,s+t),\qquad x\in\overline B,\ s\ge0.
\]
It is not hard to check that assumptions \eqref{f2}-\eqref{f3} (and also $\eqref{f1'}$ or \eqref{f1}) are inherited by $f_t$ with constants independent of $t\in[0,t_0]$.
Next, consider the family of fixed point equations
\[
u=H(u,t),\quad t\in[0,t_0],
\]
equivalently,
\[
\Delta^2 u = f(x,u+t)\ \text{in }B,\qquad u=\partial_\nu u=0\ \text{on }\partial B,\qquad u\ge0.
\]
Then, the a priori estimate proved for the original problem
applies uniformly along the family of non-linearities $f_t$ with $t \in[0,t_0]$ (see \hyperref[a priori estimate1 final]{Theorem \ref{a priori estimate1 final}} and \hyperref[a priori estimate2 final]{Theorem \ref{a priori estimate2 final}}), that is, there exists a constant $\mathcal C=\mathcal C(t_0)>0$ such that any solution
$u$ of $u=H(u,t)$ with $t\in[0,t_0]$ satisfies $\|u\|_\infty\le \mathcal C.$ Finally, setting $R:=\mathcal C+1,$ we see that there is no fixed point on the sphere $\|u\|_\infty=R$ for any $t\in[0,t_0]$. This, together with \emph{2.(b)} (no solutions for $t\ge t_0$), yields \emph{2.(c)}.

\medskip

Since all the assumptions of  \hyperref[topological method]{Theorem \ref{topological method}} are satisfied, $T$ has a fixed point $u\in C$ in the annulus $B_R\setminus\overline{B}_r$, that is
\[
u=T(u)=(\Delta^2)^{-1}\big(f(\cdot,u)\big),
\qquad r<\|u\|_\infty<R.
\]
In particular, $u\ge0$, $u\not\equiv0$, and by \hyperref[regularity1]{Theorem \ref{regularity1}} $u \in C^{3}(\overline{B})$ is a  solution of the Dirichlet biharmonic problem
\[
\Delta^2 u=f(x,u)\ \text{in }B,\qquad u=\partial_\nu u=0\ \text{on }\partial B.
\]
Moreover, since $u \not \equiv 0$, it follows that $f(x,u)\not \equiv 0$, because the only solution of
\begin{equation*}
    \Delta^2 u=0 \quad \text{in} \quad B \quad \text{and \quad } u=\partial_{\nu}u=0\quad \text{on}\quad \partial B 
\end{equation*}
is $u\equiv0.$ Thus, by the positivity of $G(x,y)$, we conclude for every $x \in B$, that
\begin{equation*}
    u(x)=\int_{B}G(x,y)f(y,u(y))\,\mathrm{d}y>0.
\end{equation*}
\end{proof}

\begin{remark} If we assume in \hyperref[existence of sol]{Theorem \ref{existence of sol}} that $f$ satisfies additionally conditions \eqref{eq:Holderx} and \eqref{eq:Lipschitzt}, we can conclude by \hyperref[rem:C4alpha]{Corollary \ref{rem:C4alpha}} that $u$ is a classical solution of problem \eqref{App}, precisely, $u \in C^4(\overline{B}).$
\end{remark}

\section{Generalization to non-variational cooperative systems}\label{system generalization}

The analysis that we developed for the scalar case extends to systems of the form
\begin{equation}\label{system-problem}\tag{S}
\begin{cases}
\Delta^2 u=g(x,v) & \text{in }B,\\
\Delta^2 v=f(x,u) & \text{in }B,\\
u=\partial_\nu u=0 & \text{on }\partial B,\\
v=\partial_\nu v=0 & \text{on }\partial B,\\
u>0,\quad v>0 & \text{in }B,
\end{cases}
\end{equation}
where $f,g:\overline{B}\times \mathds{R}^{+}\to \mathds{R}^{+}$ are continuous. A pair \((u,v)\in L^1(B)\times L^1(B)\) is a very weak solution
of \eqref{system-problem} if
\(
        f(x,u),\,g(x,v)\in L^1(B)
\)
and
\[
        \int_B u\,\Delta^2\varphi\,\mathrm{d}x
        =
        \int_B g(x,v)\varphi\,\mathrm{d}x,
        \qquad
        \int_B v\,\Delta^2\varphi\,\mathrm{d}x
        =
        \int_B f(x,u)\varphi\,\mathrm{d}x,\quad \forall\, \varphi \in C^4_{\mathrm{cl}}(\overline{B}).
\]

We establish uniform a priori bounds for the system \eqref{system-problem} with nonlinearities in two relevant growth scenarios. In the first result, we deal with nonlinearities which are Brezis-Merle subcritical:
\begin{equation}\label{system subcritical growth condition}
\exists \alpha,\beta \in (0,1), \: \exists c>0 \: \text{ such that } \: f(x,t) \leq c e^{t^{\alpha}}\hspace{0.2cm} \text{and}\hspace{0.2cm}g(x,t)\leq c e^{t^{\beta}}\quad \forall x \in B, \: t\geq 0.
\end{equation}
In the second result, we tackle Brezis-Merle critical exponential nonlinearities:
\begin{align} \label{f0System}
\exists c>0  \text{ such that }      &f(x,t) \leq c e^{t} \; \text{ and } \;  g(x,t) \leq c e^{t} \quad \forall x \in B, t \geq 0,\\ \label{system inferior growth condition}
\exists d,T>0 \text{ such that }  & de^{t}\leq  f(x,t) \; \text{ and } \; de^{t}\leq  g(x,t) \quad \forall x \in B,  \: t\geq T.
\end{align}

The above conditions \eqref{system subcritical growth condition}, \eqref{f0System}, and \eqref{system inferior growth condition} are the natural counterparts of the assumptions required throughout the previous sections to study the scalar case: \eqref{system subcritical growth condition} is equivalent to \eqref{f1'} and combining \eqref{f0System} and \eqref{system inferior growth condition} yields the system counterpart of \eqref{f1}. In particular, \eqref{f0System} is exactly \eqref{f0} with $p=1$. Therefore, the study of the non-variational cooperative system \eqref{system-problem} follows the
same general strategy developed for the scalar problem. In this section, we explain how a careful
componentwise adaptation of the scalar arguments in the previous sections extends to \eqref{system-problem}, in the spirit of the
two-dimensional elliptic systems studied in \cite{criticalandsubcritical} and \cite{joaorufdjairo2}.

First, we discuss the regularity step inspired by \hyperref[regularity1]{Theorem \ref{regularity1}} and \hyperref[rem:C4alpha]{Corollary \ref{rem:C4alpha}}. We begin by observing that condition \eqref{f0} is required for only one of the nonlinearities, $f$ or $g$. Indeed, without loss of generality, we consider the case in which this assumption is imposed on $f(x,t)$.

Since, by the definition of very weak solution $g(\cdot,v) \in L^{1}(B)$, we may apply \hyperref[integrability]{Corollary \ref{integrability}} to the Dirichlet biharmonic problem
\[
\begin{cases}
\Delta^{2}u=g(x,v) & \text{in } B,\\
u=\partial_{\nu}u=0 & \text{on } \partial B.
\end{cases}
\]
This gives $\int_{B}e^{ku}\,\mathrm{d}x<+\infty$ for every $k>0$. The growth assumption \eqref{f0} imposed on $f$ then ensures that $f(\cdot, u) \in L^{k}(B)$ for every $k>1.$ Hence, studying the problem
\[
\begin{cases}
\Delta^{2}v=f(x,u) & \text{in } B,\\
v=\partial_{\nu}v=0 & \text{on } \partial B,
\end{cases}
\]
in the same way as in the proof of \hyperref[regularity1]{Theorem \ref{regularity1}},
we obtain $v \in C^{3,\gamma}(\overline{B})\cap L^{\infty}(B).$

Since $v$ is sufficiently smooth and $g:\overline{B}\times \mathds{R}^{+}\to \mathds{R}^{+}$ is continuous, the composition $g(x,v)$ belongs to $C(\overline{B})$ and, in particular,  to $L^{k}(B)$ for every $k>1$. Therefore, proceeding once more as in \hyperref[regularity1]{Theorem \ref{regularity1}} with the problem
\[
\begin{cases}
\Delta^{2}u=g(x,v) & \text{in } B,\\
u=\partial_{\nu}u=0 & \text{on } \partial B,
\end{cases}
\]
we obtain $u \in C^{3,\gamma}(\overline{B})\cap L^{\infty}(B)$. Additionally, if we assume that both $f$ and $g$ satisfy \eqref{eq:Holderx} and \eqref{eq:Lipschitzt}, we obtain classical regularity as in \hyperref[rem:C4alpha]{Corollary \ref{rem:C4alpha}}. This leads to the following result.

\begin{theorem}\label{regularity system} Assume that either $f(x,t)$ or $g(x,t)$ satisfies \eqref{f0}. Then every very weak solution $(u,v)$ of \eqref{system-problem} satisfies $(u,v)\in(C^{3,\gamma}(\overline{B}) \cap L^{\infty}(B))^{2}$ for some $0<\gamma<1$. If we additionally assume that both  $f(x,t)$ and $g(x,t)$ satisfy  \eqref{eq:Holderx} and \eqref{eq:Lipschitzt} with the same $0< \gamma <1$, then  $(u,v)\in C^{4,\gamma}(\overline B)\times C^{4,\gamma}(\overline{B})$. Hence, in particular $(u,v)\in C^4(\overline B) \times C^4(\overline B)$.
\end{theorem}

The uniform \(L^1\) estimate is also obtained by reproducing the scalar
argument  in \hyperref[subsec4.2]{Subsection \ref{subsec4.2}} simultaneously for the two components. We first indicate the counterpart of \hyperref[moving 6]{Lemma \ref{moving 6}} for systems of the form \eqref{system-problem}. 
\begin{lemma}\label{moving 6 system}Assume that either $f$ or $g$ satisfies \eqref{f0} and that both $f$ and $g$  satisfy \eqref{f3}. Then there exists a neighborhood $\omega$ of $\partial B$ and $\delta>0$ such that  every very weak solution $(u,v)$ of \eqref{system-problem} satisfies
$$\langle \nabla u(x),\theta\rangle\leq 0\quad \text{and}\quad \langle \nabla v(x),\theta \rangle \leq0,$$
for all  $x \in \omega$ and for all unitary vectors $\theta \in \mathds{R}^4$ such that $|\theta-\nu(x)|<\delta.$
\end{lemma}
\begin{proof} We adapt the proof of \cite[Lemma 2.5]{Mancini} to the system \eqref{system-problem}. Using the same notations of \cite[Lemma 2.5]{Mancini}, we introduce the sets
\[
\Gamma_u:=
\left\{
\lambda\in(0,1):
\begin{array}{l}
u(x)>u(\overline x)\quad\forall\,x\in S_\lambda,\\[1mm]
\partial_{e_1}u(x)<0\quad\forall\,x\in T_\lambda
\end{array}
\right\},\quad\Gamma_v:=
\left\{
\lambda\in(0,1):
\begin{array}{l}
v(x)>v(\overline x)\quad\forall\,x\in S_\lambda,\\[1mm]
\partial_{e_1}v(x)<0\quad\forall\,x\in T_\lambda,
\end{array}
\right\}.
\]
where $e_1:=(1,0,0,0) \in \partial B$ and
\begin{equation*}
    \Sigma_{\lambda} \coloneq \{x \in B:x_1>1-\lambda\}, \quad T_{\lambda}\coloneq\{x \in B: x_1=1-\lambda\}, \quad S_{\lambda}\coloneq B\setminus (\Sigma_{\lambda} \cup T_{\lambda}).
\end{equation*}
Moreover, we define $\Gamma\coloneqq \Gamma_u \cap \Gamma_v$ and
$
        \overline\lambda
        :=
        \sup\{\lambda\in(0,1):\ (0,\lambda)\subset
        \Gamma\}.
$ Considering the following problems
\[
\begin{cases}
\Delta^2 u=g(x,v) & \text{in }B,\\
u=\partial_\nu u=0 & \text{on }\partial B,
\end{cases}
\qquad
\begin{cases}
\Delta^2 v=f(x,u) & \text{in }B,\\
v=\partial_\nu v=0 & \text{on }\partial B,
\end{cases}
\]
independently, we can find $\varepsilon_u, \varepsilon_v>0$ such that
\[
        (0,\varepsilon_u)\subset\Gamma_u,
        \qquad
        (0,\varepsilon_v)\subset\Gamma_v.
\]
Thus, if we set  $\varepsilon \coloneqq \min\{\varepsilon_u,\varepsilon_v\}$, we have $(0,\varepsilon)\subset \Gamma,$ showing that $\overline{\lambda}$ is well defined. We claim that
\[
        \Gamma
        \supset
        \left(0,\frac{\overline r}{2}\right),
\]
where $\overline{r}$ is given in \eqref{f3}. Suppose by contradiction that
\[
        \overline\lambda<\frac{\overline r}{2}.
\]
By the definition of \(\overline\lambda\), proceeding as in \cite[Lemma 2.5]{Mancini}, we obtain simultaneously
\[
        u(y)\ge u(\overline y),
        \qquad
        v(y)\ge v(\overline y)
        \qquad
        \forall\,y\in S_{\overline\lambda}.
\]
Moreover, by the condition \eqref{f3} and applying the strategy of  \cite[Lemma 2.5]{Mancini} twice, we also have
\[
        g(y,v(y))
        \ge
        g(y,v(\overline y))
        \ge
        g(\overline y,v(\overline y))\quad \forall\, y \in S_{\overline{\lambda}} \cap \Sigma_{\overline{\lambda}'},
\]
and
\[
        f(y,u(y))
        \ge
        f(y,u(\overline y))
        \ge
        f(\overline y,u(\overline y))\quad \forall\, y \in S_{\overline{\lambda}} \cap \Sigma_{\overline{\lambda}'},
\]
where for $x=(x_1,x_2,x_3,x_4) \in \mathds{R}^{4}$, we define
\begin{equation*}
    \Sigma_{\overline{\lambda}'}\coloneq \{x \in B: \overline{x} \in \Sigma_{\overline{\lambda}}\}\quad \text{for}\quad \overline{x}\coloneq (2(1-\lambda)-x_1,x_2,x_3,x_4). 
\end{equation*}
Next, considering the problem
\[
\begin{cases}
\Delta^2 u=g(x,v) & \text{in }B,\\
u=\partial_\nu u=0 & \text{on }\partial B,
\end{cases}
\]
we obtain
\[
        u(x)>u(\overline x)
        \quad\forall\,x\in S_{\overline\lambda},
        \qquad
        \partial_{e_1}u(x)<0
        \quad\forall\,x\in T_{\overline\lambda}.
\]
While studying the problem 
\[
\begin{cases}
\Delta^2 v=f(x,u) & \text{in }B,\\
v=\partial_\nu v=0 & \text{on }\partial B,
\end{cases}
\]
we obtain
\[
        v(x)>v(\overline x)
        \quad\forall\,x\in S_{\overline\lambda},
        \qquad
        \partial_{e_1}v(x)<0
        \quad\forall\,x\in T_{\overline\lambda}.
\]

Finally, by the same compactness argument as
in  \cite[Lemma 2.5]{Mancini}, these strict inequalities allow the plane to be
moved slightly further. Thus, there exists \(\eta>0\) such that $\overline{\lambda} +\eta<\overline{r}/2$ and
\[
        (0,\overline\lambda+\eta)
        \subset
        \Gamma,
\]
which contradicts the definition of \(\overline\lambda\).
\end{proof}

The spectral condition
\eqref{f2} imposed on both nonlinearities allows to
control both nonlinear masses through the first eigenfunction. Indeed, adapting the proof  of \cite[Lemma 3.1]{criticalandsubcritical}, we obtain the following result.

\begin{lemma}\label{L1 uniform with phi1 system} Assume that either $f$ or $g$ satisfies $\eqref{f0}$ and that both $f$ and $g$ satisfy $\eqref{f2}$. Then there exists a constant $C>0$ such that
    \begin{equation}
        \int_{B} f(x,u)\varphi_1\,\mathrm{d}x+\int_{B}g(x,v)\varphi_1\,\mathrm{d}x\leq C,
    \end{equation}
for every very weak solution $(u,v)$ of \eqref{system-problem}.
\end{lemma}
In view of \hyperref[moving 6 system]{Lemma \ref{moving 6 system}} and \hyperref[L1 uniform with phi1 system]{Lemma \ref{L1 uniform with phi1 system}}, we can easily adapt the proofs of \hyperref[Linfinity bound near the boundary]{Lemma \ref{Linfinity bound near the boundary}} and \hyperref[uniform L11]{Theorem \ref{uniform L11}} and obtain the following $L^1$ bound on the nonlinearities.
\begin{theorem}\label{uniform L11 system} Assume that $f$ or $g$ satisfies \eqref{f0}, and that both $f$ and $g$ satisfy $\eqref{f2}$, $\eqref{f3}$. Then there exists a constant $M>0$ such that
\begin{equation*}
    \int_{B} f(x,u)\,\mathrm{d}x+\int_{B}g(x,v)\,\mathrm{d}x\leq M,
\end{equation*}
for every very weak solution $(u,v)$ of \eqref{system-problem}.    
\end{theorem}
Once the \(L^1\) masses are uniformly controlled, the proofs of the
\(L^\infty\) estimates follow the same alternatives already treated in the
scalar problem \eqref{App}. Assuming the subcritical exponential regime \eqref{system subcritical growth condition} and \eqref{f2}-\eqref{f3} for both nonlinearities $f$ and $g$, the uniform control of $\lVert \Delta u\rVert_{L^{2}(B)}$ and $\lVert \Delta v\rVert_{L^{2}(B)}$ for every very weak solution $(u,v)$ of \eqref{system-problem} follows from \hyperref[uniform L11 system]{Theorem \ref{uniform L11 system}} and Adams' inequality, as in  the proof of \hyperref[Sobolev uniform estimate]{Theorem \ref{Sobolev uniform estimate}}. The main difference is that we must repeat the argument for each individual equation: 
\[
\begin{cases}
\Delta^2 u=g(x,v) & \text{in }B,\\
u=\partial_\nu u=0 & \text{on }\partial B,
\end{cases}
\qquad
\begin{cases}
\Delta^2 v=f(x,u) & \text{in }B,\\
v=\partial_\nu v=0 & \text{on }\partial B.
\end{cases}
\]
More precisely, by considering the first problem above, we obtain
\[
        \|u\|\le C(1+\|v\|^\eta),
        \qquad \text{ with }
        \eta=\frac{\beta}{2-\beta}\in(0,1) \text{ and } \beta \text{ as in \eqref{system subcritical growth condition}}.
\]
Similarly, by considering the second problem above, we obtain
\[
        \|v\|\le C(1+\|u\|^{\eta'}),
        \qquad \text{ with }
        \eta'=\frac{\alpha}{2-\alpha}\in(0,1) \text{ and } \alpha \text{ as in \eqref{system subcritical growth condition}}.
\]
Thus, iterating the two inequalities above, we have
\begin{equation}
    \lVert u\rVert\leq C(1+\lVert u\rVert^{\eta\eta'})
\quad \text{ and } \quad 
    \lVert v\rVert\leq C(1+\lVert v\rVert^{\eta\eta'}).
\end{equation}
Since $\alpha+\beta<2$, if and only if
\(
        \eta\eta'<1
\),
these two inequalities imply that, uniformly in $(u,v)$,
\[
        \|\Delta u\|_{L^2(B)}+\|\Delta v\|_{L^2(B)}\le C.
\]
Finally, adapting the proof of \hyperref[Lp uniform estimate]{Theorem \ref{Lp uniform estimate}} for both $f(x,u)$ and $g(x,v)$ yields the following uniform a priori bounds in the subcritical regime.

\begin{theorem}\label{a priori estimate2 final system} Suppose that \eqref{system subcritical growth condition} holds. In addition, suppose that both $f$ and $g$  satisfy \eqref{f2} and \eqref{f3}. Then there exists a constant $C>0$ such that  every very weak solution $(u,v)$ of \eqref{system-problem} satisfies
\begin{equation*}
    \lVert u\rVert_{L^{\infty}(B)}+\lVert v\rVert_{L^{\infty}(B)}\leq C.
\end{equation*}
\end{theorem}

In the critical exponential regime, specifically when  $f$ and $g$ satisfy \eqref{f0System} and \eqref{system inferior growth condition} and both $f$ and $g$ satisfy \eqref{f2}-\eqref{f3}, we proceed as in \hyperref[subsec a priori crit]{Section \ref{subsec a priori crit}} and argue by contradiction by considering a
sequence of positive solutions \((u_k,v_k)\) of \eqref{system-problem}
such that
\[
\|u_k\|_{L^\infty(B)}+\|v_k\|_{L^\infty(B)}\to+\infty .
\]
First, \hyperref[uniform L11 system]{Theorem \ref{uniform L11 system}} allows us, up to a
subsequence, to assume that
\[
f(x,u_k)\,\mathrm{d}x \rightharpoonup \mu,
\quad \text{ and } \quad
g(x,v_k)\,\mathrm{d}x \rightharpoonup \nu,
\]
weakly in the sense of measures. We then define the corresponding sets
of non-regular points \(\Sigma_\mu\) and \(\Sigma_\nu\), using the
Brezis-Merle threshold \(32\pi^2\), and the blow-up sets \(S_u\) and
\(S_v\) associated with the two components. We have
\[
S_u=S_v=\Sigma_\mu=\Sigma_\nu \: ;
\]
indeed, this can be achieved by combining the arguments in \cite[Section 5]{criticalandsubcritical} with the tools used in \hyperref[subsec a priori crit]{Section \ref{subsec a priori crit}}. More precisely, we apply the Brezis-Merle type inequality for the biharmonic operator instead of the Brezis-Merle inequality for the Laplacian, and the interior estimate found in \cite[Section 2.5.2]{Gazzola} in place of the mean-value property for harmonic functions.

Once this identification is obtained, the final exclusion of blow-up is
reduced to the scalar concentration argument described in \hyperref[a priori crit]{Theorem \ref{a priori crit}}; that is, the critical lower
exponential assumption \eqref{system inferior growth condition} is
combined with \hyperref[Green Below]{Lemma \ref{Green Below}} to rule out the existence of a
common blow-up point. Hence, $S_u$ and $S_v$ are empty, and the desired
uniform \(L^\infty\) bound follows for both components.

\begin{theorem}\label{a priori estimate1 final system} Suppose that \eqref{f0System} and \eqref{system inferior growth condition} hold. In addition, suppose that both $f$ and $g$  satisfy \eqref{f2} and \eqref{f3}. Then there exists a constant $C>0$ such that  every very weak solution $(u,v)$ of \eqref{system-problem} satisfies\[
\|u\|_{L^{\infty}(B)}+\|v\|_{L^{\infty}(B)}\le C.
\]
\end{theorem}

Finally, we establish the existence of solutions for problem \eqref{system-problem}, assuming \eqref{f5} with $\varepsilon_0>0$ and $\eta \in (0,6)$ for both nonlinearities $f$ and $g$, together with the hypotheses of either  \hyperref[a priori estimate2 final system]{Theorem \ref{a priori estimate2 final system}} or \hyperref[a priori estimate1 final system]{Theorem \ref{a priori estimate1 final system}}. Indeed, this can be done by applying \hyperref[topological method]{Theorem \ref{topological method}} as in \hyperref[existence]{Section \ref{existence}}, but this time in the product cone $C(\overline{B})\times C(\overline{B})$. The
operator $T$ associated with \eqref{system-problem} sends a pair \((u,v)\) to
the solution pair of the two linear Dirichlet biharmonic problems with
right-hand sides \(g(\cdot,v)\) and \(f(\cdot,u)\). Its compactness follows
directly from the compactness of each component of $T$. The smallness
assumption \eqref{f5} near the origin for \(f\) and \(g\) is used to verify condition \(1.\)
of \hyperref[topological method]{Theorem \ref{topological method}} in the product cone. The homotopy obtained by shifting both
components by the parameter \(t\) verifies condition \(2.(a)\). Condition \eqref{f2} on both $f$ and $g$ is used, as in the scalar case, to prove
condition \(2.(b)\), namely the absence of fixed points for \(t\) large. Lastly,
\hyperref[uniform L11 system]{Theorem \ref{uniform L11 system}} provides condition \(2.(c)\), since
one can choose \(R>0\) so large that no fixed point of the homotopy lies on the
sphere of radius \(R\). Therefore, \hyperref[topological method]{Theorem \ref{topological method}} applies and yields a positive
solution of the system, as stated in the following result.
\begin{theorem}\label{existence of sol system} Assume that the hypotheses of either \hyperref[a priori estimate2 final system]{Theorem \ref{a priori estimate2 final system}} or \hyperref[a priori estimate1 final system]{Theorem \ref{a priori estimate1 final system}}, and assume that both $f$ and $g$ satisfy  \eqref{f5} with $\varepsilon_0>0$ and $\eta \in (0,6)$. Then \eqref{system-problem} has at least one solution $(u,v) \in C^{3,\gamma}(\overline{B})\times C^{3,\gamma}(\overline{B})$ for some $0<\gamma<1$.
\end{theorem}

\section{A nonlinearity lacking separation of variables at infinity}\label{Example}

We conclude the paper by providing an example of a nonlinearity that satisfies the hypotheses used in our a priori estimates and existence result, while failing to satisfy the asymptotic separation condition at infinity $(H_2)$ used in \cite{Mancini}. This shows that the class of nonlinearities treated here is not merely a reformulation of that framework. Let us recall that condition $(H_2)$ in \cite{Mancini} reads as follows: there exist functions $a:B\to(0,\infty)$ and $F:[0,\infty)\to(0,\infty)$ such that
\begin{equation} \label{Mancini&RomaniH2}
\lim_{t \to + \infty} \frac{f(x,t)}{a(x)F(t)}= 1, \qquad  \text{uniformly in }x \in B.
\end{equation}

Fix $r\in \bigl( 0, \frac 1 2 \bigr)$ and set $\omega_r:=\{x\in B:\mathrm{dist}(x,\partial B)<r\}$. Choose $\psi\in C^\infty(\overline{B})$ such that
$0\le\psi\le 1$ in $B$, $\psi\equiv 0$ in $\omega_r$, and $\psi\equiv 1$ in $B_{1-2r}$. For every $x \in \overline B$, we define
\[
f(x,t):=
\begin{cases}
\eta\,t, & 0\le t\le T,\\[3pt]
\eta A_T \,e^{t}\Big(1+\varepsilon\,\psi(x)\,\sin(\log(1+t))\Big), & t> T,
\end{cases}
\]
where $\eta \in (0,6)$, $\varepsilon \in (0,1)$,
$$T:=e^{k\pi}-1>0, \quad k \geq 1 \text{ is an integer,} \quad \text{  and } \quad  A_T:= \frac T{e^T}>0.$$
By construction $f$ satisfies \eqref{f5} with $0< \varepsilon_0 \leq T$ and $\eta \in (0,6)$. Moreover, since $\eta>0$, $A_T>0$, $\varepsilon>0$, and $0\le\psi\le 1$ in $B$, we have
\begin{equation} \label{eq:ExampleCrit}
\eta A_T e^t(1- \varepsilon) \leq f(x,t) \leq \eta A_T e^t(1 + \varepsilon) \quad \forall x \in B, \: t >T.
\end{equation}
In particular, recalling the definition of $f(x,t)$ for $0 \leq t \leq T$ and combining \eqref{eq:ExampleCrit} with the fact that $\varepsilon<1$, we deduce that $f(x,t)>0$ for every $x \in B$, $t>0$. Moreover, $\sin( \log(1+T))= \sin(k \pi)=0$ and $\eta A_T e^T= \eta T$, hence $f \in C(\overline B \times [0,+ \infty))$. 

Next, we show that $f$ satisfies \eqref{f1}, \eqref{f2}, and \eqref{f3}. For every $x \in B$ and $t \in [0,T]$, we have
$$f(x,t)= \eta t \leq \eta e^t.$$
By the above estimate and \eqref{eq:ExampleCrit}, we deduce that \eqref{f1} holds with $C:=\eta \max\{1, A_T(1+ \varepsilon)\}>0$ and $D:= \eta A_T(1- \varepsilon)>0$. Using the estimate from below in \eqref{eq:ExampleCrit} and the fact that $\eta>0$, $A_T>0$, and $\varepsilon<1$, we get
$$\frac{f(x,t)}{t} \geq \frac{\eta A_T e^t(1- \varepsilon)}{t} \longrightarrow + \infty , \qquad  \text{uniformly in }x \in B,$$
and \eqref{f2} is satisfied. Since $\psi\equiv 0$ in $\omega_r$, for every $x \in \omega_r$ we have
\[
f(x,t)=
\begin{cases}
\eta\,t, & 0\le t\le T,\\[3pt]
\eta A_T \,e^{t}, & t> T,
\end{cases}
\]
and the map $t \mapsto f(x,t)$ is increasing on $[0,+ \infty)$. In addition, for every $t \geq0$, the map $x \in \omega_r \mapsto f(x, t)$ is constant; hence, $f(\cdot, t) \in C^1(\omega_r)$, $\nabla_xf(x,t)=0$ for every $x \in \omega_r$, and \eqref{f3} is satisfied.

Finally, we show that $f$ fails to satisfy the uniform asymptotic separation of variables at infinity \eqref{Mancini&RomaniH2}. Assume by contradiction that there exist functions $a:B\to(0,\infty)$ and $F:[0,\infty)\to(0,\infty)$ such that $f$ satisfies \eqref{Mancini&RomaniH2}. Since $\psi\equiv 0$ in $\omega_r$ and $\psi\equiv 1$ in $B_{1-2r}$, there exist $x_0\in\omega_r$ with $\psi(x_0)=0$ and $x_1\in B_{1-2r}$ with $\psi(x_1)=1$. Therefore, for every $t >T$, we have
\begin{equation}\label{eq:ratio}
\frac{f(x_1,t)}{f(x_0,t)}=1+\varepsilon\sin(\log(1+t)) \quad \text{ so that } \quad \nexists \lim_{t \to + \infty} \frac{f(x_1,t)}{f(x_0,t)}.
\end{equation}
On the other hand, \eqref{Mancini&RomaniH2} implies pointwise convergence for each fixed $x \in B$. In particular,
\[
\lim_{t \to + \infty }\frac{f(x_1,t)}{a(x_1)F(t)}= 1
\qquad\text{and}\qquad
\lim_{t \to + \infty}\frac{f(x_0,t)}{a(x_0)F(t)}= 1
,
\]
and hence
\[
\lim_{t \to + \infty}\frac{f(x_1,t)}{f(x_0,t)}
=\lim_{t \to + \infty}
\frac{a(x_1)}{a(x_0)}
\cdot
\frac{\displaystyle \frac{f(x_1,t)}{a(x_1)F(t)}}{\displaystyle \frac{f(x_0,t)}{a(x_0)F(t)}}
=
\frac{a(x_1)}{a(x_0)},
\]
which contradicts \eqref{eq:ratio}.

The above example covers the critical growth case \eqref{f1}. Similarly, we can also construct an example in the subcritical growth case \eqref{f1'}. Choose
\(\psi\in C^{1}(\overline B)\) such that \(0\leq \psi\leq 1\) in $B$,
\(\psi\not\equiv 0\), and \(\psi\equiv 0\) in a neighborhood of \(\partial B\). We define
\[
f(x,t)
=
A t^{p} e^{t^{\sigma}}
\left(1+\varepsilon \psi(x)\sin(\log(1+t))\right),
\qquad \forall x\in B,\ t\geq 0,
\]
where $A>0$, \(0<\sigma<1\), \(p>1\), and \(0<\varepsilon<1\).
Then \(f\) satisfies the subcritical growth condition \eqref{f1'} with $\sigma < \alpha < 1$, as well as the
structural assumptions \eqref{f2}, \eqref{f3}, and \eqref{f5}. However, as in the
previous example, the oscillating factor prevents \(f\) from satisfying the asymptotic
separation of variables at infinity \eqref{Mancini&RomaniH2}.

\vspace{1cm}



\subsection*{Statements and Declarations}

\noindent\textbf{Funding.}
This study was financed in part by the Coordenação de Aperfeiçoamento de Pessoal de Nível Superior - Brasil (CAPES) - Finance Code 001, and by the MATH-AmSud grant 88887.878894/2023-00. 
J. M. do \'O acknowledges partial support from CNPq
through grants 312340/2021-4, 409764/2023-0, and 443594/2023-6 and from the
Para\'iba State Research Foundation (FAPESQ), grant no. 3034/2021. 
F. Sani was partially supported by INdAM-GNAMPA.
V. R. Silva acknowledges partial support from CAPES through Grants 88887.824231/2023-00 and 88881.980555/2024-01.

\medskip

\noindent\textbf{Conflict of interest.}
The authors declare that they have no conflict of interest.

\medskip

\noindent\textbf{Authors' contributions.}
All authors contributed to the conception of the work, the development of the
mathematical arguments, and the writing and revision of the manuscript. All
authors read and approved the final manuscript.

\medskip

\noindent\textbf{Data availability.}
Data sharing does not apply to this article, as no datasets were generated or
analyzed during the current study.

\bigskip


\bibliographystyle{abbrv}
\bibliography{References}

\end{document}